\documentclass[11pt, english]{smfart}

\usepackage{amsmath,amsthm,sfheaders}
\usepackage{xspace,smfthm}
\usepackage[psamsfonts]{amssymb}
\usepackage[latin9]{inputenc}
\usepackage{graphicx,color}
\usepackage{hyperref,fancyhdr}
\usepackage{xypic}

\usepackage{xcolor}

\newcommand{\cO}{\mathcal{O}}

\newcommand{\Res}{\operatorname{Res}}
\newcommand{\sP}{\mathsf{P}}
\newcommand{\cS}{\mathcal{S}}
\newcommand{\diam}{\operatorname{diam}}
\newcommand{\can}{\operatorname{can}}

\newcommand{\PGL}{\mathrm{PGL}}
\newcommand{\PSU}{\mathrm{PSU}}

\newcommand{\bN}{\mathbb{N}}
\newcommand{\sF}{\mathsf{F}}
\newcommand{\sJ}{\mathsf{J}}
\newcommand{\rd}{\mathrm{d}}
\newcommand{\bA}{\mathbb{A}}
\newcommand{\bC}{\mathbb{C}}
\newcommand{\bD}{\mathbb{D}}
\newcommand{\bZ}{\mathbb{Z}}
\newcommand{\bP}{\mathbb{P}}
\newcommand{\bQ}{\mathbb{Q}}
\newcommand{\bR}{\mathbb{R}}

\newcommand{\Fix}{\operatorname{Fix}}
\newcommand{\supp}{\operatorname{supp}}

\newcommand{\Rat}{\operatorname{Rat}}
\newcommand{\Crit}{\operatorname{crit}}

\theoremstyle{plain}
\newtheorem{theorem}{\bf Theorem}[section]
\newtheorem{lemma}[theorem]{\bf Lemma}
\newtheorem{proposition}[theorem]{\bf Proposition}
\newtheorem{corollary}[theorem]{\bf Corollary} 
\newtheorem*{claim}{\bf Claim}
\newtheorem{mainth}{\bf Theorem}

\theoremstyle{definition}
\newtheorem{definition}[theorem]{\bf Definition}

\newtheorem*{caution}{\bf Caution}

\theoremstyle{remark}
\newtheorem{remark}[theorem]{\bf Remark}
\newtheorem{fact}[theorem]{\bf Fact}
\newtheorem{example}[theorem]{\bf Example} 

\def\theenumi{\roman{enumi}}
\def\labelenumi{(\theenumi)}

\numberwithin{equation}{section}

\begin{document} 

\title{Approximation of non-archimedean Lyapunov exponents and applications over global fields}
\pagestyle{fancy}
\fancyhf{}
\renewcommand{\headrulewidth}{0pt}
\lhead[\small\thepage]{\small\textsc{Approximation of non-archimedean Lyapunov exponents and applications}}
\rhead[\small\textsc{Thomas Gauthier, Y\^usuke Okuyama $\&$ Gabriel Vigny}]{\small\thepage}

\author{Thomas Gauthier}
\address{LAMFA, UPJV, 33 rue Saint-Leu, 80039 AMIENS Cedex 1, FRANCE}
\email{thomas.gauthier@u-picardie.fr}
\author{Y\^usuke Okuyama}
\address{Division of Mathematics, Kyoto Institute of Technology, Sakyo-ku, Kyoto 606-8585 JAPAN}
\email{okuyama@kit.ac.jp}
\author{Gabriel Vigny}
\address{LAMFA, UPJV, 33 rue Saint-Leu, 80039 AMIENS Cedex 1, FRANCE}
\email{gabriel.vigny@u-picardie.fr}

\thanks{The first and third authors' research is partially supported by the ANR grant Lambda ANR-13-BS01-0002.}
\thanks{The second author's research is partially supported by JSPS Grant-in-Aid 
for Scientific Research (C), 15K04924}

\date{\today}

%\subjclass[2010]{Primary 37P30; Secondary 37P45, 37F45}

%\keywords{}

\begin{abstract} 
Let $K$ be an algebraically closed field of characteristic 0 that is complete with respect to a non-archimedean absolute value. We establish a locally uniform approximation formula of the Lyapunov exponent of a rational map $f$ of $\bP^1$ of degree $d>1$ over $K$, in terms of the multipliers of $n$-periodic points of $f$, with an explicit control in terms of $n$, $f$ and $K$. As an immediate consequence, we obtain an estimate for the blow-up of the Lyapunov exponent near a pole in one-dimensional families of rational maps over $K$.

Combined with our former archimedean version, this non-archimedean quantitative approximation allows us to show:
\begin{itemize}
\item a quantified version of Silverman's and Ingram's recent comparison between the critical height and any ample height on the moduli space $\mathcal{M}_d(\bar{\mathbb{Q}})$,
\item two improvements of McMullen's finiteness of the multiplier maps: reduction to multipliers of cycles of exact given period and an effective bound from below on the period,
\item a characterization of non-affine isotrivial rational maps defined over a the function field $\mathbb{C}(X)$ of a normal projective variety $X$ in terms of the growth of the  degree of the multipliers.
\end{itemize}
\end{abstract}

\maketitle

\section{Introduction}

Fix an integer $d>1$ and let $\mathrm{Rat}_d$ be the space of degree $d$ rational maps of $\mathbb{P}^1$. This is an irreducible affine variety of dimension $2d+1$ which is defined over $\bQ$. The group linear algebraic $\mathrm{SL}_2$ acts on 
the space $\mathrm{Rat}_d$
by conjugacy. The \emph{dynamical moduli space $\mathcal{M}_d$} 
of degree $d$ rational maps on $\bP^1$
is the geometric quotient $\Rat_d/\mathrm{SL}_2$,
which is an irreducible affine variety of dimension $2d-2$ and defined over $\bQ$,
and is singular if and only if $d\geq3$ (see Silverman \cite{silverman-spacerat} for more precise statements).

\medskip

In~\cite{GOV}, the authors established an approximation formula for the Lyapunov exponent with respect to the equilibrium measure of any complex rational maps in terms of the absolute values of the multipliers of periodic points which is locally uniform on the complex manifold $\mathrm{Rat}_d(\bC)$. The aim of establishing such an approximation property was to understand the equidistribution of centers of hyperbolic components of disjoint type towards the bifurcation measure $\mu_{\mathrm{bif}}$ of the Moduli space $\mathcal{M}_d(\bC)$ and also to give a precise asymptotic count of such components, also related to the bifurcation measure.

\smallskip

Since the proof of the (complex) approximation formula essentially relies on one-dimensional potential theoretic tools and on Fatou's upper bound of attracting cycles of complex rational maps, it is natural to wonder whether a similar formula is true over any complete algebraically closed field of characteristic $0$.

~

 In the present article, we address this question and give a certain number of consequences over local and global fields.
We first present the formula and all the involved tools as well as a degeneration property over a local non-archimedean field. In a second time, we present applications over global fields.

\paragraph*{Notations}
We introduce here some notations we will use in the text which are well-defined over any field. We refer to ~\cite{Silverman} for more details.

Fix a field $k$ and let $\bar{k}$ be an algebraic closure of $k$ (we use $K$ to denote local fields in the sequel). Fix a rational map $f\in k(z)$ of degree $d>1$.

For any $n\in\mathbb{N}^*$, we denote by $\mathrm{Fix}(f^n)$ the set of all periodic points of $f$ of period dividing $n$, i.e.
$\mathrm{Fix}(f^n):=\{z\in\bP^1(\bar{k})\, : \ f^n(z)=z\}$. 
We let $\mathrm{Fix}^*(f^n)$ the set of points $z\in\mathbb{P}^1(\bar{k})$ such that
\begin{itemize}
\item either $z\in\mathrm{Fix}(f^n)\setminus \bigcup_{m|n, \, m<n}\mathrm{Fix}(f^m)$,
\item or there is $m|n$ with $m<n$ such that $z\in\mathrm{Fix}(f^m)$ and $(f^m)'(z)$ is a primitive $n/m$-root of unity.
\end{itemize}
For any $n\in\mathbb{N}^*$, we let 
\[d_n:=\sum_{m|n}\mu(n/m)(d^m+1).\]
The set $\mathrm{Fix}^*(f^n)$ contains $d_n$ points counted with multiplicity (as a point of the divisor $\{f^n(z)=z\}$). For any $1\leq j\leq d_n$, we let $\sigma_{j,n}^*(f)$ be the $j$-th elementary symmetric polynomial associated to the unordered $d_n$-tuple
$\{(f^n)'(z)\}_{z\in\mathrm{Fix}^*(f^n)}$, where points are listed with multiplicity. The map
\[\tilde\Lambda_n:f\mapsto (\sigma_{1,n}^*(f),\ldots,\sigma_{d_n,n}^*(f))\]
actually defines a morphism from the space of all degree $d$ rational maps to the affine space $\mathbb{A}^{d_n}$ which is defined over $\mathbb{Q}$ and descends to a morphism $\mathcal{M}_d\rightarrow\mathbb{A}^{d_n}$.

We also denote by $p_{d,n}(f,T)\in k[T]$ the \emph{multiplier polynomial}, i.e. the unique monic degree $d_n/n$ polynomial with coefficients in $k$ such that
\begin{gather*}
\bigl(p_{d,n}(f,T)\bigr)^n=\prod_{z\in\Fix^{*}(f^n)}\bigl((f^n)'(z)-T\bigr)=\sum_{j=0}^{d_n}\sigma_{j,n}^*(f)\cdot (-T)^{d_n-j},
\end{gather*} 
in $k[T]$, where the product over $\Fix^{*}(f^n)$
is taken with multiplicity and $\sigma_{0,n}^*\equiv1$ by convention.

Fnally, the \emph{critical set} $\mathrm{Crit}(f)$ of $f$ is the set 
\begin{gather*}
 \mathrm{Crit}(f):=\{z\in\bP^1(\bar{k}):f'(z)=0\},
\end{gather*}
of all (classical) critical points of $f$ in $\bP^1(\bar{k})$. 

\subsection{Locally uniform approximation of the Lyapunov exponent and degeneration over a non-archimedean field}

Let $K$ be an algebraically closed field of characteristic $0$
that is complete with respect to an absolute value $|\cdot|$, which 
we assume is non-trivial in that $|K|\neq\{0,1\}$.

Recall that $K$ is said to be {\itshape non-archimedean} 
if the {\itshape strong triangle inequality} 
$|z+w|\le\max\{|z|,|w|\}$
holds for any $z,w\in K$, and otherwise, to be {\itshape archimedean};
$K\cong\bC$ if and only if $K$ is archimedean.

\medskip

Assume first $K$ is either archimedean or non-archimedean. Let $f$ be any rational map of degree $d>1$ defined over $K$. For every lift $F=(F_0,F_1)$ of $f$ with $F_0(p_0,p_1)=\sum_{j=0}^da_jp_0^{d-j}p_1^j$ and 
$F_1(p_0,p_1)=\sum_{\ell=0}^db_\ell p_0^{d-\ell}p_1^\ell$, we set
\begin{gather*}
 |F|:=\max\{|a_0|,\ldots,|a_d|,|b_0|,\ldots,|b_d|\}, 
\end{gather*}
Recall that the \emph{resultant} of such a map $F$ is given by
\begin{gather*}
\Res(F):=\begin{vmatrix}
	    a_0 & \cdots & a_d  &        &    \\
	        & \ddots &      & \ddots &    \\
	        &        & a_0  & \cdots & a_d\\
	    b_0 & \cdots & b_d  &        &    \\
	        & \ddots &      & \ddots &    \\
	        &        & b_0  & \cdots & b_d\\ 
	   \end{vmatrix}\in\bZ[a_0,\ldots,a_d,b_0,\ldots,b_d]_{2d},
\end{gather*}
and that $\Res(F)\neq0$, see, e.g., \cite[\S2.4]{Silverman} for more details.
Following Kawaguchi-Silverman \cite[Definition 2]{KS09},
we say $F$ is {\itshape minimal} if $|F|=1$.
Such a minimal lift $F$ of $f$ is unique up to multiplication in $\{z\in K:|z|=1\}$,
so that the \emph{resultant of $f$}
\begin{gather*}
 \Res(f):=\Res(F)\in K^*
\end{gather*}
where $F$ is a minimal lift of $f$, is defined
up to multiplication in $\{z\in K:|z|=1\}$.  We may identify $\Rat_d(K)$ with the Zariski open set $\mathrm{Rat}_d(K)=\bP^{2d+1}(K)\setminus\{\mathrm{Res}=0\}$.
When $K$ is non-archimedean,
by the strong triangle inequality, $|\Res F|\le 1=|F|$.

Let $\sP^1:=\sP^1(K)$ be the Berkovich projective line over $K$. The Lyapunov exponent of $f$ with respect to the equilibrium (or canonical) measure $\mu_f$
 of is defined by
$L(f):=\int_{\sP^1}\log f^\#\mu_f\in\bR$, where $f^\#$ is the continuous extension to $\sP^1$ of the chordal derivative of $f$ (see Section~\ref{sec:Lyap} for more details). For every $n\in\bN^*$, every $r\in]0,1]$ and every 
$f\in\Rat_d(K)$, set
\begin{gather*}
 L_n(f,r):=\frac{1}{nd_n}\sum_{z\in\Fix^*(f^n)}\log\max\{r,|(f^n)'(z)|\}.
\end{gather*}
For $n\in\mathbb{N}^*$, let 
$\sigma_2(n):=\sum_{m\in\bN:m|n}m^2$. Note that $\sigma_2(n)=O\bigl(n^2\log\log n\bigr)$
as $n\to\infty$.

\medskip

Suppose now $K$ is non-archimedean. For every integer $d>1$ and
every $n\in\bN^*$, following \cite{BIJL14}, we also set
\begin{gather*}
 \epsilon_{d^n}:=\min\bigl\{|m|^{d^n}:m\in\{1,2,\ldots,d^n\}\bigr\}\in|K^*|\subset]0,1].
\end{gather*}
Our first result, from which we derive several applications, is the following non-archimedean counterpart of \cite[Theorem 3.1]{GOV}, where we adapt our strategy to the non-archimedean setting.

\begin{mainth}\label{tm:approx}
Let $K$ be an algebraically closed field of characteristic $0$
that is complete with respect to a non-trivial and non-archimedean absolute value $|\cdot|$, and fix an integer $d>1$.  
Then for every $f\in\Rat_d(K)$, every $n\in\bN^*$, 
and every $r\in]0,\epsilon_{d^n}]$,
\begin{align*}
\bigl|L_n(f,r)-L(f)\bigr|
\le 8(d-1)^2\Bigg(|L(f)| & -\frac{4d^2-2d-1}{d(2d-2)}\log|\mathrm{Res}(f)|+|\log r|\Bigg)\frac{\sigma_2(n)}{d_n}.
\end{align*}
In particular, $|Res(f)|$ can be replaced with $\inf_{h\in[f]} |Res(h)|$ in the above inequality.
\end{mainth}

This type of results has a long history. For a fixed $f$ defined over a number field, Szpiro and Tucker~\cite{SzpiroTucker} proved a qualitative version of this approximation formula. In~\cite{okuyama:speed}, the second author proved a quantitative version of this formula with an error term which does not depend nicely   on the map.

\medskip

As usual, let $\cO(\mathbb{D}_K)$ be 
the ring of $K$-analytic functions on $\mathbb{D}_K=\{z\in K:|z|<1\}$, i.e.
a function $f\in \cO(\mathbb{D}_K)$ if it can be expanded as a power series $f(t)=\sum_{i\geq0}a_it^i$ with the condition $|a_i|\rho^i\rightarrow0$ for all $\rho<1$  since $K$ is non-archimedean.
Observe that a function $f$ belongs to $\cO(\mathbb{D}_K)[t^{-1}]$ iff it is a meromorphic function on $\mathbb{D}_K$ with no pole in $\mathbb{D}_K\setminus\{0\}$.
An element $f_t(z)\in \cO(\mathbb{D}_K)[t^{-1}](z)$
is said to be a {\itshape meromorphic family} of
rational functions of degree $d$ parametrized by $\mathbb{D}_K$ if we have
$f_t\in\Rat_d(K)$ for every $t\in\mathbb{D}_K\setminus\{0\}$.
The following consequence of Theorem \ref{tm:approx}
is a non-archimedean counterpart of 
a combination of \cite[Theorem~1.4]{DeMarco2} and \cite[Proposition~3.1]{DeMarco-stable} 
(see also \cite[Theorem 3.6]{GOV}).

\begin{mainth}[Degeneration]\label{tm:degenerate}
Let $K$ be an algebraically closed field of characteristic $0$
that is complete with respect to a non-trivial and
non-archimedean absolute value $|\cdot|$, and fix an integer $d>1$.
Then for every meromorphic family $f_t(z)\in\cO(\mathbb{D}_K)[t^{-1}](z)$ 
of rational functions of degree $d$ parametrized by $\mathbb{D}_K$,
there exists $\alpha\in\mathbb{R}_+$ such that
\begin{gather*}
 L(f_t)=\alpha\log|t|^{-1}+o(\log|t|^{-1})\quad\text{as }t\to 0.
\end{gather*}
\end{mainth}

Note that Favre~\cite{Favre-degeneration} has proved the same statement for meromorphic families of endomorphisms of $\mathbb{P}^k(\mathbb{C})$ using the machinery of \emph{hybrid spaces}.
Note also that, in the case of families of polynomials, the much more stronger assertion that the (subharmonic) function $t\mapsto L(f_t)-\alpha\log|t|^{-1}$ on $\mathbb{D}_K^*$
extends continuously around $t=0$ has been established 
in \cite{conti} by Favre and the first author.
Note that the continuity statement is not true in full generality. Indeed, DeMarco and the second author~\cite{disconti} provide a family of examples of meromorphic families of complex rational maps for which $L(f_t)-\alpha\log|t|^{-1}$ is unbounded near $0$.

\subsection{Global fields and height functions}
A {\itshape global field} is a field $k$ which is canonically equipped with
\begin{enumerate}
\item a proper set $M_k$ of all {\itshape places} of $k$,
\item a family $(|\cdot|_v)_{v\in M_k}$,
where for each $v\in M_k$, $|\cdot|_v$ is a non-trivial absolute value
of $k$ representing $v$. A place $v\in M_k$ is said to be an {\itshape infinite place}
of $k$ if $|\cdot|_v$ is archimedean; otherwise, $v\in M_k$ is said to be 
a {\itshape finite place} of $k$,
\item a family $(N_v)_{v\in M_k}$ of positive integers
\end{enumerate}
such that for every $z\in k^\times$, we have $|z|_v=1$
for all but finitely many places $v\in M_k$ and the so called \emph{product formula} holds:
\[\prod_{v\in M_k}|z|_v^{N_v}=1.\]
Note in addition that if $k'$ is an algebraic extension of $k$, then the set of places $w\in M_{k'}$ that extend places of $v$, denoted by $w|v$, is finite and that they satisfy the compatibility condition
\[\sum_{w|v}N_w=[k':k].\]
Number fields are the archetype of global fields. In this paragraph, we give applications of Theorem~\ref{tm:approx1} over number fields.

\medskip

For each $v\in M_k$, let $k_v$ be the completion
of $k$ with respect to $|\cdot|_v$ and $\bC_v$ the completion of an algebraic
closure of $k_v$ with respect to $|\cdot|_v$, and
we fix an embedding of $\overline{k}$ in $\bC_v$ which extends that of $k$.
By convention, the dependence of any local quantity
induced by $|\cdot|_v$ on each $v\in M_k$ is emphasized 
by adding the suffix $v$ to it, e.g., we will denote $L(f)_v$ for the Lyapunov exponent of $f\in\mathrm{Rat}_d(\bar{k})$ acting 
on $\sP^1(\bC_v)$.

\medskip 

Choose an integer $N\geq1$.
The \emph{naive height} function $h_{\mathbb{P}^N,k}$
on $\bP^N(\overline{k})$ associated to $k$ 
is defined for every $x=[x_0:\ldots:x_n]\in\mathbb{P}^N(\bar{k})$ by
\begin{gather*}
 h_{\mathbb{P}^N,k}(x)
=\frac{1}{[k':k]}\sum_{v\in M_{k'}}N_v\log\max_{0\leq j\leq N}|x_j|_v,
\end{gather*}
where $k'$ is any finite extension of $k$ 
such that $x\in\mathbb{P}^N(k')$.

\medskip

Let $k$ be a global field, and fix an integer $d>1$. We define a height function $h_{d,k}$
on $\Rat_d(\bar{k})$ as follows: we identify $\mathrm{Rat}_d(\bar{k})$ with $\bP^{2d+1}(\bar{k})\setminus\{\mathrm{Res}=0\}$, which is a Zariski open subset of
$\bP^{2d+1}(\overline{k})$. For every $f=[a_0,\ldots,a_d,b_0,\ldots,b_d]\in\Rat_d(\bar{k})$, we let
\[
h_{d,k}(f):=h_{\mathbb{P}^{2d+1},k}([a_0,\ldots,a_d,b_0,\ldots,b_d]).
\]

\smallskip

Pick $f\in\Rat_d(k)$. The \emph{Call--Silverman canonical height function} $\hat{h}_{f,k}$ of $f$
on $\bP^1(\overline{k})$ {\itshape relative to} $k$ is defined by
\begin{gather*}
\hat{h}_{f,k}(z):=\lim_{n\to\infty}\frac{h_{\bP^1,k}(f^n(z))}{d^n}\in\bR,
\quad z\in\bP^1(\overline{k}).
\end{gather*}
The \emph{critical height} function $h_{\operatorname{crit},k}$
on $\Rat_d(\overline{k})$ relative to $k$ is defined by
\begin{gather*}
h_{\operatorname{crit},k}(f)
:=\frac{1}{[k':k]}\sum_{c\in\mathrm{Crit}(f)}\hat{h}_{f,k'}(c),
\quad f\in\Rat_d(\overline{k}),
\end{gather*}
where $k'$ is any finite extension of $k$
such that $f\in\Rat_d(k')$ and $\mathrm{Crit}(f)\subset\bP^1(k')$.
The function $h_{\operatorname{crit},k}$ on $\Rat_d(\overline{k})$ descends to a function on $\mathcal{M}_d(\overline{k})$, which is still denoted by $h_{\operatorname{crit},k}$.

Although the critical height functions defined above are named as ``height'' functions, it is a priori not clear at all whether they are height functions, in any possible sense.

\subsection{Applications over number fields}

Fix $d>1$. Silverman conjectured~\cite[Conjecture 6.29]{silverman-moduli} that the absolute critical height $h_{\mathrm{crit},\mathbb{Q}}$ is commensurable to any ample height on $\mathcal{M}_d$ away from the flexible Latt\`es locus. Ingram~\cite{Ingram-critheight} recently proved this conjecture. We give here a quantitative version of Ingram's result.

\medskip

Let us be more specific. Recall that $f\in\Rat_d(\mathbb{C})$ is a \emph{Latt\`es map}
if there exists a complex elliptic curve $E$ and a homomorphism $\phi:E\to E$ and a finite morphism $\pi: E\to \bP^1(\mathbb{C})$ such that 
\[f\circ \pi=\pi\circ \phi \ \text{ on} \ E.\]
We say $f$ is flexible if $\phi$ is given by $\phi(P)=[m]P+P_0$ for some $m\in\mathbb{Z}$ and some $P_0\in E$. Let $\mathcal{L}_d\subset\mathcal{M}_d$ be the flexible Latt\`es locus; it is an affine curve which is defined over $\mathbb{Q}$.

We also fix an embedding $\iota:\mathcal{M}_d\hookrightarrow\bA^N$
of the moduli space in some affine space and let $\overline{\mathcal{M}}_d$ be
the Zariski closure of $\iota(\mathcal{M}_d)$ in the ambient projective space $\bP^N$. The line bundle $D:=\mathcal{O}_{\mathbb{P}^N}(1)|_{\overline{\mathcal{M}}_d}$ defines an ample line bundle on $\overline{\mathcal{M}}_d$.
Let also 
\[h_{\mathcal{M}_d,D}:=h_{\mathbb{P}^N,\mathbb{Q}}|_{\iota(\mathcal{M}_d)}=h_{\mathbb{P}^N,\mathbb{Q}}\circ\iota+O(1),\]
and set
\[C_1(d,D):=\frac{1}{8(d-1)}\left(\frac{\|\mu_{\mathrm{bif}}\|_{\mathcal{M}_d}}{\deg_D(\mathcal{M}_d)}\right)^{1/{(2d-2)}} \ \ \text{and} \ \ C_2(d,D)=2(d-1)\frac{\|T_\mathrm{bif}\|_{\mathcal{M}_d,D}}{\deg_D(\mathcal{M}_d)},\]
where $T_\mathrm{bif}$ and $\mu_{\mathrm{bif}}$ are the \emph{bifurcation current} and the \emph{bifurcation measure} (see Definition~\ref{mu_and_T} below) and ${\|T_\mathrm{bif}\|_{\mathcal{M}_d,D}}$ and $\|\mu_{\mathrm{bif}}\|_{\mathcal{M}_d}$ denote their respective mass. Of course, the mass of $T_\mathrm{bif}$ depends on the choice of the embedding. We prove

\begin{mainth}\label{tm:critheight}
For every integer $d>1$, there exists a constant $A\ge 0$ such that
\begin{gather*}
 C_1(d,D)h_{\mathcal{M}_d,D}-A\leq h_{\operatorname{crit},\bQ}\leq C_2(d,D)h_{\mathcal{M}_d,D}+A
\end{gather*}
on $(\mathcal{M}_d\setminus\mathcal{L}_d)(\bar{\bQ})$. Moreover, the constant $A$ can be computed explicitly.
\end{mainth}

The strategy of the proof is similar to that of Ingram~\cite{Ingram-critheight}. In his proof, he relates the critical height to the height of a single multiplier. Instead, using our quantitative approximation formulas of $L(f)_v$ 
for each place, we relate the critical height with the average of the heights of \emph{all} multipliers of cycles of period dividing a given $n$. Recall that
\[\tilde\Lambda_n:\mathcal{M}_d\to\mathbb{A}^{d^n+1}\]
denote the map which, to any conjugacy class $[f]$ associates the symmetric functions of the multipliers of periodic of period dividing $n$.

One of the interests of our strategy is that our relation between the critical and multiplier heights can be used in both ways.
As an example of this idea, we get that the multipliers of formal exact period $n$ define a moduli height for all $n$ large enough: we prove that for all $n\geq n_1$, there exists a computable constant $C$ depending only on $n$ and $d$ such that
\[\frac{C_1(d,D)}{2}h_{\mathcal{M}_d,D}-C\leq \frac{1}{nd_n}h_{\mathbb{P}^{d_n},\bQ}\circ\tilde\Lambda_{n} \leq \left(C_2(d,D)+\frac{C_1(d,D)}{2}\right)h_{\mathcal{M}_d,D}+C,\]
on $(\mathcal{M}_d\setminus\mathcal{L}_d)(\bar{\mathbb{Q}})$. Using classical finiteness properties of ample height functions, we prove

\begin{mainth}[Multipliers of exact period $n$ as moduli]\label{tm:McMexact}
There exists an integer $n_1\geq1$ which depends only on $d$, the mass $\|\mu_{\mathrm{bif}}\|_{\mathcal{M}_d}$ and $\|T_{\mathrm{bif}}\|_{\mathcal{M}_d,D}$ and the degree $\deg_D(\mathcal{M}_d)$ and such that for all $n\geq n_1$, the exact multiplier map $\tilde\Lambda_{n}:\mathcal{M}_d(\mathbb{C})\rightarrow\mathbb{C}^{d_n}$ is finite-to-one outside of the curve of flexible Latt\`es locus $\mathcal{L}_d(\mathbb{C})$.
\end{mainth}

Theorem~\ref{tm:McMexact} improves McMullen's finiteness theorem\cite[Corollary 2.3]{McMullen4} which states that, for every integer $d>1$, if $n$ is large enough,
the map
\[\Lambda_n:\mathcal{M}_d(\mathbb{C})\to\mathbb{C}^{d^n+1}\]
given by the symmetric functions of the multipliers of periodic points $z\in \mathrm{Fix}(f^n)$ of period dividing $n$
is finite-to-one on $(\mathcal{M}_d\setminus\mathcal{L}_d)(\mathbb{C})$.

\subsection{Applications over global function fields}
Let $X$ be a connected smooth projective curve define over $\bC$ and let $k:=\bC(X)$ be the function field of $X$.
In this case, there is a one-to-one correspondence between closed points of $X$ and places of the field $k$. If $v\in M_k$ corresponds to the point $x$, we choose
$|\cdot|_v$ to be defined by $|g|_v=e^{-\mathrm{ord}_x(g)}$ for all $g\in \bC(X)$.  
Moreover, we can take $N_v=1$ for every $v\in M_k$. The naive height function on $\mathbb{P}^N(\overline{k})$ associated to $k$ can be computed as
\[
h_{\mathbb{P}^N,k}(x)=\max_{0\leq j\leq N}\frac{\deg(x_j)}{\deg(\pi)},
\quad x=[x_0:\cdots:x_N],
\]
where $\pi:Y\rightarrow X$ is any finite morphism such that
that $x\in\mathbb{P}^N(\bC(Y))$.
 In particular, the canonical height function $\hat{h}_f$ of a rational map $f$ defined over $k$ can be computed as
\begin{gather*}
x\mapsto\hat{h}_f(x)=\lim_{n\to\infty}\frac{1}{d^n}\frac{\deg(f^n(x))}{\deg(\pi_x)}, 
\end{gather*}
where $\pi_x:Y\rightarrow X$ is any finite branched cover such that $x\in\mathbb{P}^1(\mathbb{C}(Y))$.

We say $f\in\Rat_d(k)$ is \emph{isotrivial} 
if there exists a finite branched cover $Y\to X$ and
a M\"obius transformation $M\in\PGL(2)$ which is defined over the finite extension $k'=\bC(Y)$ of $k$ such that the rational map 
$M\circ f\circ M^{-1}$ is actually defined over $\bC$. This is equivalent to the fact that the specialization $(M\circ f\circ M^{-1})_t\in\Rat_d(\bC)$ are
independent of $t\in Y$. 

\smallskip

We say that $f\in\Rat_d(k)$ is \emph{affine}
if there exists a non-isotrivial elliptic curve $E$ defined over $k$ and a homomorphism $\phi:E\to E$ given by $\phi(P)=[m]P+P_0$ for some $m\in\mathbb{Z}$ and some $P_0\in E$ and a degree $2$ morphism $E\to \bP^1_k$ such that
\[f\circ\pi=\pi\circ f\ \text{ on} \ E.\]
Equivalently, we say $f$ is affine if 
for every $t\in X$, the specialization $f_t\in\Rat_d(\bC)$ is a Latt\`es map, i.e.
if there exists a complex elliptic curve $E_t$, a double branched cover $\pi_t:E_t\to\bP^1(\bC)$ and an integer $m\in\mathbb{Z}$ and $P_{0,t}\in E_t$ such that $\pi_t([m]P+P_{0,t})=f_t\circ\pi_t(P)$ on $E$. In particular, $d=m^2$.

Applying Theorem~\ref{tm:approx} to the action of $f$ on $\sP^1(\bC_v)$ 
for all $v\in M_k$, together with
previously known results, we easily establish the following quantitative
characterization of those exceptional properties for such a rational map $f$.

\begin{mainth}\label{mainthm:functionfield}
Let $X$ be a connected smooth projective curve defined over $\bC$, and
fix an integer $d>1$. For every $f\in\Rat_d(\bC(X))$,
the following assertions are equivalent:
\begin{enumerate}
\item $f$ is either isotrivial or affine,
\item $\max_{1\leq j\leq d_n}\deg(\sigma_{j,n}^*(f))=o(nd^n)$ as $n\rightarrow\infty$,
\item $\max_{1\leq j\leq d_n}\deg(\sigma_{j,n}^*(f))=O(n\sigma_2(n))$ as $n\rightarrow\infty$.
\end{enumerate}
\end{mainth}

Note finally that a similar statement over the function field of a normal irreducible projective variety defined over $\mathbb{C}$ can be obtained, applying the same strategy word by word, see~\ref{sec:higherdim}.

\subsection*{Structure of the paper}

The paper is organized as follows: In Section~2 we recall classical facts about the dynamics of rational maps over any metrized field and give preliminary technical results. 
Section~3 is dedicated to the proof of Theorem~\ref{tm:approx} and~\ref{tm:degenerate}.
We then focus on estimates over global function fields in Section~4. Section~5 is concerned with the proof of Theorem~\ref{mainthm:functionfield}.
The applications of the archimedean approximation formula are given in Section~6 and Section~7 and~8 are dedicated to the proofs of Theorem~\ref{tm:critheight} and~\ref{tm:McMexact}.

\subsection*{Acknowledgment}
We would like to thank Patrick Ingram for useful discussions and for providing the reference~\cite{BombieriGubler}, and Nguyen-Bac Dang for useful discussions concerning Siu's numerical bigness criterion. The second author also thanks the LAMFA and the Universit\'e de Picardie Jules Verne, where he was visiting in autumn 2017 and this work grew up.

\section{Background and key algebraic estimates}

\subsection{The Berkovich projective line and the dynamics of rational maps}
Let $K$ be an algebraically closed field that is complete with respect to
a non-trivial absolute value $|\cdot|$. When $K$ is non-archimedean, 
let $\cO_K$ be the ring of $K$-integers $\cO_K:=\{z\in K:|z|\le 1\}$. The residual characteristic of $K$ is by definition
the characteristic of the residue field of $\cO_K$.

Let $\|\cdot\|$ be the maximal norm
on $K^2$ (when $K$ is non-archimedean) or the Hermitian norm on $K^2$
(when $K$ is archimedean, i.e., $K\cong\bC$).
In the rest of this section, 
for simplicity we denote $\bP^1(K)$ by $\bP^1$.
With the wedge product $(p_0,p_1)\wedge(q_0,q_1):=p_0q_1-p_1q_0$ on $K$,
the {\itshape normalized} chordal metric $[z,w]$ on $\bP^1$
is defined by 
\begin{gather*}
 [z,w]:=\frac{|p\wedge q|}{\|p\|\cdot\|q\|},\quad\text{where }p\in\pi^{-1}(z),q\in\pi^{-1}(w), 
\end{gather*}
on $\bP^1$ (and as a function on $\bP^1\times\bP^1$).
The subgroup $U_K$ in the linear fractional transformation group $\PGL(2,K)$ 
defined by $\PGL(2,\cO_K)$ (when $K$ is non-archimedean) or $\PSU(2,K)$
(when $K$ is archimedean) 
acts isometrically on $(\bP^1,[z,w])$.

Let us recall some details on 
the {\itshape Berkovich} projective line $\sP^1=\sP^1(K)$.
When $K$ is archimedean, then $\sP^1$ coincides with $\bP^1$, so 
we assume $K$ is non-archimedean until the end of this paragraph. Then
$\sP^1$ is a compact, Hausdorff,
and uniquely arcwise connected topological space augmenting $\bP^1$,  
and is also a tree in the sense of Jonsson \cite[Definition 2.2]{Jonsson15}.
A point $\cS\in\sP^1$ is of one and only one of types I, II, III, and IV. 
Any type I, II, or III point $\cS\in\sP^1\setminus\{\infty\}$
is represented by the multiplicative supremum seminorm 
$|\cdot|_{\{z\in K:|z-a|\le r\}}$ on $K[z]$
(or, by a $K$-closed disk
\begin{gather*}
B_{\cS}:=\{z\in K:|z-a|\le r\}
\end{gather*}
in $K$ itself),
and setting $\diam\cS:=r$, the point $\cS$ is of type I 
if and only if $\diam\cS=0$ (or equivalently $\cS\in K=\bP^1\setminus\{\infty\}$), 
and is of type II (resp.\ III) if and only if $\diam\cS\in|K^*|$ 
(resp.\ $\diam\cS\in\bR_{>0}\setminus|K^*|$). 
The point $\infty\in\bP^1$ is also of type I, and
the {\itshape canonical $($or Gauss$)$} point $\cS_{\can}$ in $\sP^1$
is of type II and is represented by $\cO_K$ (similarly,
type IV points are {\itshape represented} by a decreasing sequence $(B_n)$ of
$K$-closed disks in $K$ such that $\bigcap_nB_n=\emptyset$).
The tree structure on $\sP^1$ is induced by the
partial ordering of all those disks in $K$ by inclusions.
The topology of $\sP^1$ is nothing but the weak topology of $\sP^1$ as a tree.
For every $P\in K[z]$, the function $|P|$ on $K$ extends continuously 
to $\sP^1\setminus\{\infty\}$ so that 
\begin{gather*}
|P|(\cS)=\sup_{B_{\cS}}|P(\cdot)|\quad\text{on }\sP^1\setminus\{\infty\};
\end{gather*}
in particular $|P|(\cS_{\can})=\max\{|\text{coefficients of }P|\}$
by the strong triangle inequality. 

When $K$ is non-archimedean,
the chordal distance $[z,w]$ on $\bP^1(\times\bP^1)$
uniquely extends to a (jointly) upper semicontinuous and
separately continuous function $[\cS,\cS']_{\can}$
on $\sP^1\times\sP^1$,
which is called the {\itshape generalized Hsia kernel on $\sP^1$
with respect to $\cS_{\can}$} (for more details, see \cite[\S 3.4]{FRL}, \cite[\S 4.4]{BR10}); for every $z\in\bP^1$, $[\cS_{\can},z]_{\can}=1$, and
for any $\cS,\cS',\cS''\in\sP^1$, we still have the strong triangle inequality
\begin{gather*}
[\cS,\cS']_{\can}\le\max\{[\cS,\cS'']_{\can},[\cS',\cS'']_{\can}\}
\end{gather*}
(see e.g.\ \cite[Proposition 4.10(C)]{BR10}). When $K$ is archimedean,
it is convenient to define the kernel function $[z,w]_{\can}$ on $\bP^1(\times\bP^1)$by $[z,w]$ itself. 

For every $\cS\in\sP^1$, let
$\delta_{\cS}$ be the Dirac measure on $\sP^1$ at $\cS$.
Let us introduce the probability Radon measure
\begin{gather*}
 \Omega_{\can}:=\begin{cases}
		\delta_{\cS_{\can}}&\text{when }K\text{ is non-archimedean}\\
		\omega&\text{when }K\text{ is archimedean}
	       \end{cases}\quad\text{on }\sP^1,
\end{gather*}
where $\omega$ is the Fubini-Study area element on $\bP^1$ normalized as
$\omega(\bP^1)=1$ when $K$ is archimedean,
and for every $\cS\in\sP^1$.
The Laplacian $\Delta$ on $\sP^1$ is normalized so that for each $\cS'\in\sP^1$, 
\begin{gather*}
\Delta\log[\cdot,\cS']_{\can}=\delta_{\cS'}-\Omega_{\can}\quad\text{on }\sP^1
\end{gather*}
(for non-archimedean $K$, see \cite[\S5.4]{BR10}, \cite[\S2.4]{FR09}; 
in \cite{BR10} the opposite sign convention on $\Delta$ is adopted).

The action on $\bP^1$ of each $f\in K(z)$
extends to a canonical continuous action on $\sP^1$. 
If in addition $\deg f>0$, then
the canonical action on $\sP^1$ of $f$ is also surjective, open, and
finite. Moreover,
the local degree function $a\mapsto\deg_a f$ of $f$ on $\bP^1$ also canonically
extends to an upper semicontinuous
function $\sP^1\to\{1,2,\ldots,\deg f\}$ such that for every domain $V\subset\sP^1$
and every component $U$ of $f^{-1}(V)$, the function 
$\cS\mapsto\sum_{\cS'\in f^{-1}(\cS)\cap U}\deg_{\cS'}(f)$ is constant on $V$.
The pullback action $f^*$ induced by $f$ on the space of Radon measures on $\sP^1$
is defined for every $\cS\in\sP^1$ by
$f^*\delta_{\cS}=\sum_{\cS'\in f^{-1}(\cS)}\deg_{\cS'}(f)\cdot\delta_{\cS'}$
on $\sP^1$, and satisfies
the functoriality $f^*\circ\Delta=\Delta\circ f^*$ (see e.g.\ \cite[\S9]{BR10}). 

For every $f\in K(z)$ of degree $>1$,
the $f$-{\itshape equilibrium $($or canonical$)$ measure} 
$\mu_f$ on $\sP^1$ is defined by
such a unique probability Radon measure $\nu$ on $\sP^1$ that 
\begin{gather*}
f^*\nu=(\deg f)\cdot \nu\text{ on }\sP^1\quad\text{and}\quad \nu(E(f))=0,
\end{gather*}
where $E(f):=\{a\in\bP^1:\#\bigcup_{n\in\bN}f^{-n}(a)<\infty\}$
is at most countable and, if $K$ has characteristic $0$, then
$\#E(f)\le 2$. The (Berkovich) Julia set $\sJ(f)$ is defined by
the support of $\mu_f$; a point $\cS$ in $\sP^1$ is
in $\sJ(f)$ if and only if for any sequence $(n_j)$ in $\bN$
tending to $\infty$ as $j\to\infty$, we have
\begin{gather*}
\bigcap_{U:\text{ open in }\sP^1\text{ and contains }\cS}
\Biggl(\bigcup_{j\in\bN}f^{n_j}(U)\Biggr)=\sP^1\setminus E(f)
\end{gather*}
(see \cite[\S10]{BR10}, \cite[\S2]{ACL}, \cite[\S3.1]{FR09}
for more details). 
The (Berkovich) Fatou set $\sF(f)$ is the open
subset $\sP^1\setminus\sJ(f)$ in $\sP^1$, and each component of $\sF(f)$ is called
a (Berkovich) Fatou component of $f$.
If $W$ is a (Berkovich) Fatou component of $f$, then
so is $f(W)$, and 
we say $W$ is cyclic under $f$ if $f^N(W)=W$ for some $N\in\bN^*$. 
We also say that $W$ is an immediate Berkovich attracting basin of period $N$ for $f$ if
there is a $($super$)$attracting fixed point $a$ of $f^N$ in
$W\cap\bP^1$ with $\lim_{n\to\infty}(f^N)^n(\cS)=a$ for any $\cS\in W$.

\subsection{The dynamical Green function}
A {\itshape continuous weight $g$ on} $\sP^1$ is
a continuous function on $\sP^1$ such that
$\mu^g:=\Delta g+\delta_{\cS_{\can}}$ 
is a probability Radon measure on $\sP^1$.
For a continuous weight $g$ on $\sP^1$,
the $g$-{\itshape potential kernel} on $\sP^1$
(or the {\itshape negative of} an Arakelov Green kernel function on $\sP^1$
relative to $\mu^g$ \cite[\S 8.10]{BR10} or an $\Omega_{\can}$-quasipotential on $\sP^1$ of $\mu^g$) is defined by
\begin{gather*}
\Phi_g(\cS,\cS'):=\log[\cS,\cS']_{\can}-g(\cS)-g(\cS')
\end{gather*}
on $\sP^1\times\sP^1$, the {\itshape $g$-equilibrium energy}
$V_g$ {\itshape of $\sP^1$} (in fact $V_g\in\bR$)
is the supremum of the $g$-energy functional
$\nu\mapsto\int_{\sP^1\times\sP^1}\Phi_g\rd(\nu\times\nu)\in[-\infty,+\infty)$
over all probability Radon measures $\nu$ on $\sP^1$, and we call
a probability Radon measure $\nu$ on $\sP^1$ at which
the above $g$-energy functional attains the supremum $V_g$
a {\itshape $g$-equilibrium mass distribution on} $\sP^1$;
in fact, $\mu^g$ is the unique $g$-equilibrium mass distribution on $\sP^1$,
and 
\begin{gather*}
 \cS\mapsto\int_{\sP^1}\Phi_g(\cS,\cdot)\rd\mu^g\equiv V_g
\end{gather*}
on $\sP^1$
(for non-archimedean $K$, see \cite[Theorem 8.67, Proposition 8.70]{BR10}).
A {\itshape normalized weight $g$ on} $\sP^1$ 
is a continuous weight on $\sP^1$ satisfying $V_g=0$;
for every continuous weight $g$ on $\sP^1$, $\overline{g}:=g+V_g/2$ is
the {\itshape unique} normalized weight on $\sP^1$ 
such that $\mu^{\overline{g}}=\mu^g$ on $\sP^1$.

From now on, pick $f\in K(z)$ of degree $d>1$.
For every lift $F$ of $f$, the function 
\begin{gather*}
T_F:=\log\|F(\cdot)\|-d\cdot\log\|\cdot\|
\end{gather*}
on $K^2\setminus\{0\}$ descends to $\bP^1$ and in turn continuously extends 
to $\sP^1$, and the function $T_F/d$ is a continuous weight on $\sP^1$; indeed, 
$\Delta(T_F/d)+\Omega_{\can}=(f^*\delta_{\Omega_{\can}})/d$ 
is a probability Radon measure on $\sP^1$ 
(see, e.g., \cite[Definition 2.8]{Okucharacterization} for non-archimedean $K$).

For every $n\in\bN^*$, $F^n$ is a lift of $f^n$, and
by a telescoping sum argument (cf.\ \cite[\S 10]{BR10}, \cite[\S 6.1]{FR09}), 
the following limit exists
\begin{gather}
g_F:=\lim_{n\to\infty}\frac{T_{F^n}}{d^n}
=\lim_{n\to\infty}\sum_{j=0}^{n-1}\frac{T_F\circ f^j}{d^{j+1}}\quad\text{on }\sP^1.\label{eq:defGreen}
\end{gather}
More precisely, for every $n\in\bN^*$, we have the following uniform error estimate
\begin{gather}
\sup_{\sP^1}\left|g_F-\frac{T_{F^n}}{d^n}\right|\le\frac{\sup_{\sP^1}|T_F|}{d^n(d-1)},\label{eq:telescope}
\end{gather}
and for each $\alpha\in K\setminus\{0\}$, 
\begin{gather}
T_{\alpha F}=T_F+\log|\alpha|\quad\text{and}\quad
g_{\alpha F}=g_F+\frac{\log|\alpha|}{d-1}.\label{eq:scaling} 
\end{gather}
on $\sP^1$. The function $g_F$ 
is called the {\itshape dynamical Green function on $\sP^1$ of} 
$F$, and which is a continuous weight on $\sP^1$ satisfying $\mu^{g_F}=\mu_f$
on $\sP^1$; indeed, 
$\mu^{g_F}=\Delta g_F+\Omega_{\cS_{\can}}
=\lim_{n\to\infty}((f^n)^*\Omega_{\cS_{\can}}/d^n)$
weakly on $\sP^1$, and $\mu^{g_F}$ satisfies the defining two properties of $\mu_f$.
We note the following {\itshape pull-back formula}
\begin{gather}
g_F\circ f=d\cdot g_F-T_F\quad\text{on }\sP^1\label{eq:pullback}
\end{gather}
(\cite[Proof of Lemma 2.4]{okuyama:speed}) and
the following {\itshape energy} formula
\begin{gather*}
V_{g_F}=-\frac{\log|\Res F|}{d(d-1)}
\end{gather*}
(for a simple proof, see e.g. \cite[Appendix A]{Baker09} or \cite[Appendix]{OS11}.
For $K\cong\bC$, this is equivalent to 
DeMarco's capacity formula \cite{DeMarco2}).
More intrinsically,
the {\itshape dynamical Green function $g_f$ of $f$ on} $\sP^1$ is 
defined by the unique normalized weight on $\sP^1$ satisfying $\mu^{g_f}=\mu_f$
on $\sP^1$; for every lift $F$ of $f$, we have
\begin{gather}
g_f=g_F-\frac{\log|\Res F|}{d(2d-2)}\quad\text{on }\sP^1,\label{eq:differencegreen}
\end{gather}
so that 
\begin{gather*}
\Phi_{g_f}=\Phi_{g_F}+\frac{\log|\Res F|}{d(d-1)}\quad\text{on }\sP^1\times\sP^1.
\end{gather*}
Since $\Res F\in K\setminus\{0\}$ and by the homogeneity of $F$, up to multiplication in $\{z\in K:|z|=1\}$, there is 
a unique lift $F$ of $f$ such that 
$g_F=g_f$ on $\sP^1$, or equivalently, that $|\Res F|=1$.

\begin{caution}
Normalizing a lift $F$ of $f$ as
$|\Res F|=1$ is as natural as normalizing $F$ as $|F|=1$.
In this article, we are saying $F$ to be {\itshape minimal} 
(rather than ``normalized'') if it satisfies the latter condition $|F|$, 
and we would avoid to use the (more ambiguous)
terminology ``a normalized lift of $f$'' for both normalizations of $F$.
\end{caution}

We recall a standard fact from elimination theory.

\begin{fact}\label{factresultant}
 Let $R$ be an integral domain. Then for any two homogeneous polynomials
 $F_0(X,Y)=\sum_{j=0}^da_jX^{d-j}Y^j$,
 $F_1(X,Y)=\sum_{k=0}^db_kX^{d-k}Y^k$ in $R[X,Y]_d$,
 there exist $G_1,G_2,H_1,H_2\in R[X,Y]_{d-1}$ such that
 \begin{align*}
 F_1(X,Y)G_1(X,Y)+F_2(X,Y)G_2(X,Y)=&(\Res F)X^{2d-1}\quad\text{and}\\
 F_1(X,Y)H_1(X,Y)+F_2(X,Y)H_2(X,Y)=&(\Res F)Y^{2d-1}
 \quad\text{in }R[X,Y]
 \end{align*}
 and that all the $4d$ coefficients of $G_1,G_2,H_1,H_2\in R[X,Y]_{d-1}$
 (determined by $(a_0,\ldots,a_d,b_0,\ldots,b_d)$) 
 are in $\bZ[a_0,\ldots,a_d,b_0,\ldots,b_d]_{2d-1}$
 (see e.g.\ \cite[Proposition 2.13]{Silverman}). 
\end{fact}

The following algebraic estimates on $T_F$ and $g_F$ are crucial
for our purpose.

\begin{lemma}\label{lm:growthlift}
There exists a constant $A_1(d,K)\geq0$ depending only on $d$ and $K$ such that for every lift $F$ of any $f\in\Rat_d(K)$, on $\sP^1$,
\begin{eqnarray*}
 \log|\Res F|-(2d-1)\log|F|-A_1(d,K) \le & T_F & \leq\log|F|+A_1(d,K)\\
\log|\Res F|-(2d-1)\log|F|-A_1(d,K)\le & (d-1)g_F & \le \log|F|+A_1(d,K).
\end{eqnarray*}
Moreover, $A_1(d,K)=0$ when $K$ is non-archimedean.
\end{lemma}

\begin{proof}
Fix $f\in\Rat_d(K)$ and a lift $F=(F_0,F_1)$ of $f$,
and write as $F_0(X,Y)=\sum_{j=0}^da_jX^jY^{d-j}$ and
$F_1(X,Y)=\sum_{k=0}^db_kX^kY^{d-k}$ and let $G_1,G_2,H_1,H_2\in K[X_0,X_1]_{d-1}$ be  the homogeneous polynomials associated to $F_0,F_1$ as in Fact~\ref{factresultant}. Set
\begin{gather*}
 A:=\begin{cases}
      1 &\text{when }K\text{ is non-archimedean},\\
      \sqrt{2(d+1)} &\text{when }K\text{ is archimedean}.
     \end{cases}
\end{gather*}
Fix $p\in K^2$. By the (strong) triangle inequality, we have
\begin{gather*}
 \|F(p)\|\le A|F|\cdot\|p\|^d.
\end{gather*}
When $K$ is non-archimedean, using the strong triangle inequality and Fact~\ref{factresultant} gives
\begin{gather*}
 \max\{|G_1(p)|,|G_2(p)|,|H_1(p)|,|H_2(p)|\}\leq |F|^{2d-1}\|p\|^{d-1}.
\end{gather*}
When $K$ is archimedean, the triangle inequality and Fact~\ref{factresultant} give
\begin{gather*}
 \max\{|G_1(p)|,|G_2(p)|,|H_1(p)|,|H_2(p)|\}\leq A'|F|^{2d-1}\|p\|^{d-1}
\end{gather*}
where $A'$ depends only on $d$ is given by $2d$ times the maximum of the absolute value of the integer coefficients that appear in the expression of the coefficients of $G_1,G_2,H_1,H_2$ as polynomials in $a_0,\ldots,a_d,b_0,\ldots,b_d$.

Using the equalities in Fact~\ref{factresultant}, we get
\[
 |\Res F|\cdot\|p\|^{d}\leq |F|^{2d-1}\|F(p)\|
\]
if $K$ is non-archimedean and
\[
 |\Res F|\cdot\|p\|^{d}\leq A''|F|^{2d-1}\|F(p)\|
\]
if $K$ is archimedean, where $A''=2\sqrt{2} A'$ and the lemma follows by definition of $T_F$ and $g_F$.
\end{proof}

\subsection{The chordal derivative and a lemma \`a la Przytycki}\label{sec:Lyap}

Pick $f\in K(z)$ of degree $d>1$.
The {\itshape chordal derivative} of $f$ is
the function
\[z\mapsto f^\#(z):=\lim_{\bP^1\ni w\to z}\frac{[f(w),f(z)]}{[w,z]}\in\mathbb{R}_+\]
on $\bP^1(K)$, so that 
\begin{gather*}
f^\#(z)=|f'(z)|\quad\text{if }f(z)=z.
\end{gather*}
For every lift $F$ of $f$, by a computation involving Euler's identity, we have
\begin{gather}
f^\#(z)=\frac{1}{|d|}\frac{\|p\|^2}{\|F(p)\|^2}|\det DF(p)|,
\quad\text{where }p\in\pi^{-1}(z),\label{eq:derivative}
\end{gather}
on $\bP^1$ (cf.\ \cite[Theorem 4.3]{Jonsson98}). 
We can note that the map $f:\bP^1\to\bP^1$ is $M_1(f)$-Lipschitz continuous for the metric $[\cdot,\cdot]$, where
\begin{gather*}
 M_1(f):=\begin{cases}
	  1/|\Res(f)| & \text{when }K\text{ is non-archimedean (by Rumely-Winburn \cite{RumelyWinburn15})},\\
	  \sup_{\bP^1}(f^\#) & \text{when }K\text{ is archimedean},
	 \end{cases}
\end{gather*}
and let us set $M(f):=M_1(f)^2(\ge 1)$.

The following is a non-archimedean counterpart to \cite[Lemma~3.1]{distribGV} (see also~\cite[Lemma 2.2]{GOV}).

\begin{lemma}\label{lm:Przytycki}
Let $K$ be an algebraically closed field of characteristic $0$
that is complete with respect to a non-trivial and non-archimedean absolute value $|\cdot|$.
For any $f\in K(z)$ of degree $d>1$, any $c\in\mathrm{Crit}(f)$ and any $n\in\bN^*$, if $0<[f^n(c),c]<1/M_1(f)^{n+1}$, then
$[\sJ(f),c]_{\can}\ge 1/M_1(f)^{n+1}$. Moreover, 
$c$ lies in the Berkovich immediate basin of a
$($super$)$attracting fixed point $a\in\bP^1(K)$ of $f^n$ 
such that $[a,c]\le[f^n(c),c]$.
\end{lemma}

\begin{proof}
Fix $f\in K(z)$ of degree $>1$ and $c\in \mathrm{Crit}(f)$.
We claim that for every $z\in\bP^1$, if $[z,c]<1/M_1(f)$, then
\begin{gather}
[f(z),f(c)]\le M_1(f)^2\cdot[z,c]^2;\label{eq:classical}
\end{gather}
indeed, fixing such $A,B\in\PGL(2,\cO_K)$ that $A(f(c))=B^{-1}(c)=0$ and
setting $g:=A\circ f\circ B$, we have $g(0)=g'(0)=0$,
so that $h(z):=g(z)/z$ extends analytically to $z=0$ and that $h(0)=0$, 
and $M_1(f)=M_1(g)$. If $|z|<1/M_1(g)(\le 1)$, then we have
\begin{gather*}
|g(z)|=[g(z),0]=[g(z),g(0)]\le M_1(g)\cdot[z,0]=M_1(g)|z|<1,
\end{gather*}
and in particular $|h(z)|\le M_1(g)$.
Hence by the (non-archimedean) Schwarz lemma (see \cite[Corollaire 1.4 $\&$ Lemme de Schwarz]{Juan03}),
if $0<|z|<1/M_1(g)$, then
\begin{gather*}
\left|\frac{g(z)}{z}\right|=|h(z)-h(0)|
\le\frac{\sup_{|w|<1/M_1(g)}|h(w)|}{1/M_1(g)}|z-0|\le M_1(g)^2|z|,
\end{gather*} 
that is, if $[z,0](=|z|)<1/M_1(g)$, then
$[g(z),0]=|g(z)|\le M_1(g)^2|z|^2=M_1(g)^2[z,0]^2$. Hence the claim holds.

Then by \eqref{eq:classical}, the separate continuity of $[\cS,\cS']_{\can}$
on $\sP^1\times\sP^1$, the density of $\bP^1$ in $\sP^1$, 
and the continuity of the action on $\sP^1$ of $f$,
for every $\cS\in\sP^1$, if $[\cS,c]_{\can}<1/M_1(f)$, then we still have
$[f(\cS),f(c)]_{\can}\le M_1(f)^2\cdot[\cS,c]_{\can}^2$.
Similarly, for every $\cS\in\sP^1$ and every $n\in\bN$, we have
$[f^n(\cS),f^n(c)]_{\can}\le M_1(f)^{n-1}\cdot[f(\cS),f(c)]_{\can}$.

Fix $n\in\bN$ so that $0<[f^n(c),c]<1/M_1(f)^{n+1}$.
We claim that $[\sJ(f),c]_{\can}\ge 1/M_1(f)^{n+1}$; otherwise,
for every $\cS\in\sP^1$, if
$[\cS,c]_{\can}<\epsilon_0:=\max\{[f^n(c),c],[\sJ(f),c]_{\can}\}
(\in(0,1/M_1(f)^{n+1}))$, then by \eqref{eq:classical}, we have
\begin{gather*}
[f^n(\cS),f^n(c)]_{\can}
\le M_1(f)^{n-1}\cdot[f(\cS),f(c)]_{\can}
\le M_1(f)^{n+1}\cdot[\cS,c]_{\can}^2
\le[\cS,c]_{\can}<\epsilon_0,
\end{gather*}
so that 
$[f^n(\cS),c]_{\can}\le\max\{[f^n(\cS),f^n(c)]_{\can},[f^n(c),c]\}<\epsilon_0$.
Hence setting 
$B(c,\epsilon_0):=\{\cS\in\sP^1:[\cS,c]_{\can}<\epsilon_0(<1)\}$,
which is a domain in $\sP^1$ satisfying
$\#(\sP^1\setminus B(c,\epsilon_0))>2$,
we must have $f^{n\cdot j}(B(c,\epsilon_0))\subset B(c,\epsilon_0)$
for every $j\in\bN$, so that $B(c,\epsilon_0)\subset\sP^1\setminus\sJ(f)$.
However, by the choice of $\epsilon_0$, we have
$B(c,\epsilon_0)\cap\sJ(f)\neq\emptyset$,
which is a contradiction. Hence the claim holds.

Replacing the above $\epsilon_0$ with $\epsilon_1:=[f^n(c),c]$, setting $B(c,\epsilon_1):=\{\cS\in\sP^1:[\cS,c]_{\can}<\epsilon_1(<1)\}$,
we still have 
$f^n(B(c,\epsilon_1))\subset B(c,\epsilon_1)\subset\sF(f)$.
Let $W$ be the Berkovich Fatou component of $f$ containing
$B(c,\epsilon_1)$, which is cyclic since we have $f^n(W)=W$. 
Fix such $C\in\PGL(2,\cO_K)$ that $C(c)=0$. Then
$\{z\in K:|z|\le\epsilon_1\}=C(B(c,\epsilon_1)\cap\bP^1)
\supset(C\circ f^n\circ C^{-1})(\{z\in K:|z|\le\epsilon_1\})$. 
Hence by the (non-archimedean) Schwarz lemma 
(see \cite[Corollaire 1.4]{Juan03}),
$C\circ f^n\circ C^{-1}$ has a fixed point $b$ in 
$\{z\in K:|z|\le\epsilon_1\}=\{z\in K:|z-b|\le\epsilon_1\}$,
so that
$|(C\circ f^n\circ C^{-1})'(b)|\le 1$, and 
indeed $|(C\circ f^n\circ C^{-1})'(b)|< 1$
since $C^{-1}(0)=c\in\operatorname{Crit}(f^n)$.

Hence
$a:=C^{-1}(b)\in B(c,\epsilon_1)\cap\bP^1\subset W\cap\bP^1$
is a (super) attracting fixed point of $f^n$ and
$[a,c]=[C(a),C(c)]=[b,0]=|b|\le\epsilon_1=[f^n(c),c]$.
\end{proof}

By an argument similar to the above shows a more effective version
of \cite[Lemma~3.1]{distribGV}, which would be needed below.

\begin{lemma}\label{th:Parchi}
Let $K$ be an algebraically closed field of characteristic $0$
that is complete with respect to an archimedean absolute value.
For every $f\in K(z)$ of degree $>1$, every $c\in \mathrm{Crit}(f)$,
and every $n\in\bN^*$, 
if $0<[f^n(c),c]<1/(64M_1(f)^{n+1})$, then
$[J(f),c]\ge 1/(32M_1(f)^{n+1})$ and, moreover, 
$c$ lies in the immediate basin of a
$($super$)$attracting fixed point $a\in\bP^1(K)$ of $f^n$ 
such that $[a,c]\le 2[f^n(c),c]$.
\end{lemma}

The following algebraic upper estimate on $f^\#$ is a keystone in the sequel.

\begin{lemma}\label{th:upperlyap}
Let $K$ be an algebraically closed field of characteristic $0$ that is complete with respect to a non-tivial absolute value $|\cdot|$. There exists a constant $A_2(d,K)\geq0$ depending only on $d$ and $K$ such that for every $f\in\Rat_d(K)$ and every lift $F$ of $f$, 
\begin{align*}
 \log(f^\#)
\le-\log|d|-2\log|\Res F|+4d\log|F|+A_2(d,K) \quad\text{on }\sP^1.
\end{align*}
Moreover, $A_2(d,K)=0$ when $K$ is non-archimedean. 
\end{lemma}

\begin{proof}
Fix $f\in\Rat_d(K)$ and a lift $F=(F_0,F_1)$ of $f$. 
Then $\det DF(X,Y)\in K[X,Y]_{2(d-1)}$ 
and, writing $F_0=\sum_{j=0}^da_jX^jY^{d-j}$ and
$F_1=\sum_{k=0}^db_kX^kY^{d-k}$, all the coefficients of 
$\det F$ are in $\bZ[a_0,\ldots,a_d,b_0,\ldots,b_d]_2$. 
By the (strong) triangle inequality,
we have
$|\det DF(p)|\leq C|F|^2$ on $\{p\in K^2:\|p\|\le 1\}$,
where we set $C:=1$ when $K$ is non-archimedean or
$C:=2d^2(d+1)^2$ when $K$ is archimedean. Combined with \eqref{eq:derivative}, this yields
$\log(f^\#)\circ\pi(p)\le-\log|d|-2T_F\circ\pi+2\log|F|+\log C$
on $\{p\in K^2:\|p\|\le 1\}$, so that also by Lemma~\ref{lm:growthlift},
we have
\[\log(f^\#)\le-\log|d|-2T_F+2\log|F|+\log C
\le-\log|d|-2\log|\Res F|+4d\log|F|+A_2(d,K)\]
on $\bP^1$, and in turn on $\sP^1$
by the continuity of the chordal derivative $f^\#$ 
on $\bP^1$ and the density of $\bP^1$ in $\sP^1$. 
\end{proof}

\subsection{The Lyapunov exponent}

From now on, let us also assume that $K$ has characteristic $0$.
Pick $f\in K(z)$ of degree $d>1$ and choose $C_1^F,\ldots,C_{2d-2}^F\in K^2\setminus\{0\}$
so that the Jacobian determinant $\det DF\in K[p_0,p_1]$ of a lift $F$ of $f$ 
factors as 
\begin{gather*}
 \det DF(p_0,p_1)=\prod_{j=1}^{2d-2}((p_0,p_1)\wedge C_j^F)
\end{gather*}
in $K[p_0,p_1]$ 
and setting $c_j:=\pi(C_j^F)$ for each $j\in\{1,\ldots,2d-2\}$(so that $\{c_j:j\in\{1,\ldots,2d-2\}\}=\mathrm{Crit}(f)$), by \eqref{eq:derivative},
the function 
$f^\#$ on $\bP^1$ extends continuously to $\sP^1$ so that
\begin{align*}
\log(f^\#)=&-\log|d|+\sum_{j=1}^{2d-2}(\log[\cdot,c_j]_{\can}+\log\|C_j^F\|)-2T_F\\
=&-\log|d|+\sum_{j=1}^{2d-2}(\Phi_{g_F}(\cdot,c_j)+g_F(c_j)+\log\|C_j^F\|)
+2g_F\circ f-2g_F\quad\text{on }\sP^1.
\end{align*}

\begin{definition}[\cite{ FR09}]
The {\itshape Lyapunov exponent} 
of $f$ with respect to $\mu_f$ is
\begin{align*}
L(f):=\int_{\sP^1}\log(f^\#)\rd\mu_f\in\mathbb{R}.
\end{align*}
\end{definition}

Note that $L(f)$ can be written as
\begin{align}
L(f)=-\log|d|+\sum_{j=1}^{2d-2}(g_F(c_j)+\log\|C_j^F\|),\label{eq:Lyapdef}
\end{align}
where $F$ is a lift of $f$ satisfying $|\Res F|=1$, so that
\begin{gather}
\log(f^\#)=L(f)+\sum_{c\in \mathrm{Crit}(f)}\Phi_{g_f}(\cdot,c)
+2g_f\circ f-2g_f\quad\text{on }\sP^1,
\label{eq:formula}
\end{gather}
and where the sum over $\mathrm{Crit}(f)$ takes into account the multiplicity 
of each $c\in \mathrm{Crit}(f)$
(\cite[Lemma 3.6]{Okuyama:FeketeMZ} or \cite[Lemma 2.4]{okuyama:speed}).
Note that $L(f)$ needs not be positive in full generality.

The following lower estimate
is crucial for our purpose.

\begin{proposition}\label{tm:nonarchLyap}
Let $K$ be an algebraically closed field of characteristic $0$
that is complete with respect to a non-trivial absolute value $|\cdot|$, and fix an integer $d>1$. 
Then for every $f\in\Rat_d(K)$ and every lift $F$ of $f$, 
we have
\begin{align*}
L(f)\geq-\log|d|
+\sum_{c\in \mathrm{Crit}(f)}g_F(c)-\frac{2}{d}\log|\Res F|+\log|C_F|,
\end{align*}
where $C_F\in K\setminus\{0\}$ is the 
leading coefficient of $\det DF(1,z)\in K[z]$.
\end{proposition}

\begin{proof}
Fix a lift $F$ of $f$ and pick $C_1^F,\ldots,C_{2d-2}^F$ as above.
Then by \eqref{eq:Lyapdef} and by the homogeneity property \eqref{eq:differencegreen} of $g_F$,
we have
\begin{align*}
L(f)=&-\log|d|+\sum_{j=1}^{2d-2}(g_F(c_j)+\log\|C_j^F\|)-\frac{2}{d}\log|\Res F|\\
\ge&-\log|d|+\sum_{c\in\Crit(f)}g_F(c_j)+\sum_{j:c_j\neq\infty}\log|C_j^F\wedge(0,1)|+\sum_{j:c_j=\infty}\log|C_j^F\wedge(1,0)|\\
&-\frac{2}{d}\log|\Res F|,
\end{align*}
where we used that $\|(0,1)\|=\|(1,0)\|=1$. The fact that
\begin{gather*}
\log|C_F|=\sum_{j:c_j\neq\infty}\log|C_j^F\wedge(0,1)|+\sum_{j:c_j=\infty}\log|C_j^F\wedge(1,0)|
\end{gather*}
completes the proof.
\end{proof}

\section{Locally uniform approximation of the Lyapunov exponent}\label{sec:approx}

In the whole section, $K$ is an algebraically closed field of characteristic $0$ that is complete with respect to a non-trivial absolute value $|\cdot|$.
Fix an integer $d>1$. For every $n\in\bN^*$ and every $r>0$, define the function
\begin{gather}
f\mapsto L_n(f,r):=\frac{1}{nd_n}\sum_{z\in\Fix^{*}(f^n)}\log\max\{r,|(f^n)(z)|\}\label{eq:truncate}
\end{gather}
on $\Rat_d(K)$, where the sum over $\Fix^{*}(f^n)$ is taken wit multiplicity.

When $K$ is archimedean, i.e., $K\cong\bC$,
the function $L_n(\cdot,r)$ is nothing but the
function 
\[f\mapsto \frac1{d_n}\int_0^{2\pi}\log|p_{d,n}(f,re^{i\theta})|\frac{\rd\theta}{2\pi}\]
on $\Rat_d(K)$.
When $K$ is non-archimedean, for every $f\in\Rat_d(K)$ and every $n\in\bN^*$,
the function $w\mapsto|p_{d,n}(f,w)|$ on $K$ extends continuously to $\sP^1\setminus\{\infty\}$. Denote by $|p_{d,n}(f,\cdot)|$ this extension.
Moreover, for every $r>0$, let 
$\cS_{0,r}$ be the point in $\sP^1$ such that $B_{\cS_{0,r}}=\{z\in K:|z|\le r\}$.
By Gauss lemma, when $K$ is non-archimedean, 
for every $f\in\Rat_d(K)$, every $n\in\bN^*$, and every $r>0$, we have
\begin{gather}
L_n(f,r)=\frac{1}{d_n}\int_{\sP^1}\log|p_{d,n}(f,\cdot)|\rd\delta_{\cS_{0,r}}
=\frac{1}{n d_n}\log\max_{0\leq j\leq d_n}|\sigma_{j,n}^*(f)|\cdot r^{d_n-j}.\label{eq:heighttruncate}
\end{gather}

\subsection{The non-archimedean locally uniform approximation formula}

The following is a more precise version of Theorem~\ref{tm:approx}.

\begin{theorem}\label{tm:approx1}
Let $K$ be an algebraically closed field of characteristic $0$
that is complete with respect to a non-trivial
and non-archimedean absolute value, and fix an integer $d>1$.
Then for every $f\in\Rat_d(K)$, every $n\in\bN^*$, 
and every $r\in(0,\epsilon_{d^n}]$,
\begin{align*}
\bigl|L_n(f,r)-L(f)\bigr|
\le 2(2d-2)^2\cdot\left(|L(f)|+\log M(f)+\sup_{\sP^1}|g_f|+|\log r|\right)\frac{\sigma_2(n)}{d_n}.
\end{align*}
\end{theorem}

\begin{proof}
Fix an integer $d>1$, $n\in\bN^*$, and $r\in(0,\epsilon_{d^n}]$.
In the following, for each $f\in\Rat_d(K)$ and every $m\in\bN^*$,
the sums over subsets in $\mathrm{Crit}(f^m)$, $\Fix(f^m)$, or $\Fix^*(f^m)$ 
take into account the multiplicities of their elements.

Fix $f\in\Rat_d(K)$ satisfying $\Fix(f^n)\cap \operatorname{Crit}(f)=\emptyset$.
By the M\"obius inversion, we have
\begin{align}
\notag d_n\cdot L_n(f,r)=&\frac{1}{n}\sum_{z\in\Fix^*(f^n)}\log\max\{r,(f^n)^\#(z)\}\\
\notag=&\frac{1}{n}\sum_{m|n}\mu\left(\frac{n}{m}\right)
\sum_{z\in\Fix(f^m)}\frac{n}{m}\log\max\{r^{m/n},(f^m)^\#(z)\}\\
=&\sum_{m|n}\mu\left(\frac{n}{m}\right)\cdot
\frac{1}{m}\sum_{z\in\Fix(f^m)}\log\max\{r^{m/n},(f^m)^\#(z)\}.\label{eq:Mobius}
\end{align}
For every $m\in\bN^*$ dividing $n$, taking the sums of 
both sides in \eqref{eq:formula} over $\Fix(f^m)$, 
by $\sum_{z\in\Fix(f^m)}1=d^m+1$ and \cite[Lemma 3.5]{OkuyamaStawiska}, we have
\begin{align*}
\frac{1}{m}\sum_{z\in\Fix(f^m)}\log((f^m)^\#(z))
=&\sum_{z\in\Fix(f^m)}\log((f^\#)(z))\\
=&(d^m+1)\cdot L(f)+\sum_{c\in \mathrm{Crit}(f)}\sum_{z\in\Fix(f^m)}\Phi_{g_f}(z,c)\\
=&(d^m+1)\cdot L(f)+\sum_{c\in \mathrm{Crit}(f)}\Phi_{g_f}(f^m(c),c)
\end{align*}
and similarly,
\begin{align*}
\frac{1}{m}\sum_{z\in\Fix(f^m):(f^n)^\#(z)<r}\log((f^m)^\#(z))
= & \, \mathrm{Card}\{z\in\Fix(f^m):(f^n)^\#(z)<r\}\cdot L(f)\\
 & \, +\sum_{c\in \mathrm{Crit}(f)}\sum_{z\in\Fix(f^m):(f^n)^\#(z)<r}\Phi_{g_f}(z,c),
\end{align*}
so that
\begin{multline}
\frac{1}{m}\sum_{z\in\Fix(f^m)}\log\max\{r^{m/n},(f^m)^\#(z)\}\\ 
=\left(d^m+1-\#\{z\in\Fix(f^m):(f^n)^\#(z)<r\}\right)\cdot L(f)\\
+\sum_{c\in \mathrm{Crit}(f)}\biggl(\Phi_{g_f}(f^m(c),c)-\sum_{z\in\Fix(f^m):(f^n)^\#(z)<r}\Phi_{g_f}(z,c)\biggr)\\
+\frac{\#\{z\in\Fix(f^m):(f^n)^\#(z)<r\}}{n}\cdot\log r.\label{eq:approx}
\end{multline}

We claim that, for every $c\in \mathrm{Crit}(f)$ and every $m\in\bN^*$ dividing $n$, 
\begin{multline}
\Biggl|\log[f^m(c),c]-\sum_{z\in\Fix(f^m):(f^m)^\#(z)<r}\log[z,c]\Biggr|\\
\le 
\#\{z\in\Fix(f^m):(f^n)^\#(z)<r\}\cdot(2m+1)\log M_1(f);\label{eq:distance}
\end{multline}
for,
if $[f^m(c),c]\ge 1/M_1(f)^{m+1}$, then 
$\log[f^m(c),c]-\sum_{z\in\Fix(f^m):(f^m)^\#(z)<r}\log[z,c]\ge 
-(m+1)\log M_1(f)+0(\le 0)$.
On the other hand, for every $z\in\Fix(f^m)$, we have
$[f^m(c),c]\le\max\{[f^m(c),f^m(z)],[z,c]\}
\le M_1(f)^m\cdot [z,c]$, 
so that
\begin{multline*}
\log[f^m(c),c]-\sum_{z\in\Fix(f^m):(f^m)^\#(z)<r}\log[z,c]\\
\le 0+\#\{z\in\Fix(f^m):(f^n)^\#(z)<r\}\cdot(2m+1)\log M_1(f).
\end{multline*}
Hence the claim holds in this case. If $(0<)[f^m(c),c]<1/M_1(f)^{m+1}$, then
by Theorem \ref{lm:Przytycki}, we have
$[c,\sJ(f)]_{\can}\ge 1/M_1(f)^{m+1}$ and
$c$ lies in the Berkovich immediate basin of a
(super)-attracting fixed point $a\in\bP^1$ of $f^m$ 
such that $[a,c]\le[f^m(c),c]$. In particular,
$0\le\log[f^m(c),c]-\log[a,c]$, so that
$0\le\log[f^m(c),c]-\sum_{z\in\Fix(f^m):(f^m)^\#(z)<r}\log[z,c]$.
On the other hand, we not only have
\begin{gather*}
[f^m(c),c]\le\max\{[f^m(c),f^m(a)],[a,c]\}\le M_1(f)^m\cdot [a,c] 
\end{gather*}
but, noting that for 
each classical attracting fixed point $z$ of $f^m$
other than $a$, the unique arc joining $c$ and $z$ in $\sP^1$ 
intersects $\sJ(f)$, also have
\begin{gather*}
\inf_{z\in\Fix(f^m)\setminus\{a\}:(f^m)^\#(z)<r}[c,z]
\ge[c,\sJ(f)]_{\can}\ge 1/M_1(f)^{m+1},
\end{gather*}
so that
\begin{gather*}
\log[f^m(c),c]-\sum_{z\in\Fix(f^m) \atop (f^m)^\#(z)<r}\log[z,c]
\le \#\{z\in\Fix(f^m):(f^n)^\#(z)<r\}\cdot(m+1)\log M_1(f). 
\end{gather*}
Hence the claim also holds in this case.

\medskip

 Once \eqref{eq:Mobius}, \eqref{eq:approx}, and \eqref{eq:distance} 
are at our disposal, we have 
\begin{multline*}
\left|L_n(f,r)-\left(1-\frac{\sum_{m\in\bN:m|n}\mu(n/m)\#\{z\in\Fix(f^m):(f^n)^\#(z)<r\}}{d_n}\right)L(f)\right|\\
\le \frac{2d-2}{d_n}\cdot\biggl(\sum_{m\in\bN:m|n}m\cdot\#\{z\in\Fix(f^m):(f^n)^\#(z)<r\}\cdot(2m+1)\log M_1(f)\\
+\sum_{m\in\bN:m|n}(2+\#\{z\in\Fix(f^m):(f^n)^\#(z)<r\})\cdot\sup_{\sP^1}|g_f|\biggr)\\
+\frac{\sum_{m\in\bN:m|n}\#\{z\in\Fix(f^m):(f^n)^\#(z)<r\}}{n\cdot d_n}\cdot|\log r|.
\end{multline*}
Now, we recall that, by Benedetto-Ingram-Jones-Levy's non-archimedean counterpart to Fatou's upper bound of the number of (super)attracting cycles \cite[Theorem 1.5]{BIJL14}, 
for every $m\in\bN^*$ and every $0<r\leq \epsilon_{d^m}$,
\begin{gather}
\#\{z\in\Fix(f^m):(f^m)^\#(z)<r\}\le\#\mathrm{Crit}(f)\cdot m\le(2d-2)\cdot m.\label{eq:attracting}
\end{gather}
Hence we have the desired inequality under the assumption that
$\Fix(f^n)\cap \mathrm{Crit}(f)=\emptyset$. 

By \eqref{eq:telescope} and Lemma~\ref{lm:growthlift}, 
the function $(f,z)\mapsto g_f(z)$ on $\Rat_d(K)\times\bP^1$ is continuous,
and then by \eqref{eq:formula}, the function $f\mapsto L(f)$ on $\Rat_d(K)$
is also continuous. Moreover, by \eqref{eq:heighttruncate},
the function $f\mapsto L_n(f,r)$  on $\Rat_d(K)$ is continuous.
Hence the left hand side in the desired inequality depends continuously
on $f\in\Rat_d(K)$, and also by the definition of $M(f)$, 
the right hand side depends upper semicontinuously on $f\in\Rat_d(K)$. 
Since $\{f\in\Rat_d(K):\Fix(f^n)\cap \mathrm{Crit}(f)\neq\emptyset\}$ 
is a proper Zariski closed subset in $\Rat_d(K)$,
the desired inequality still holds
without assuming $\Fix(f^n)\cap \mathrm{Crit}(f)=\emptyset$.
\end{proof}

An argument similar to the above shows a little more effective version 
of \cite[Theorem 3.1]{GOV}, which is needed below.

\begin{theorem}\label{prop:approxArch}
Let $K$ be an algebraically closed field of characteristic $0$
that is complete with respect to an archimedean absolute value $|\cdot|$, 
and fix an integer $d>1$.
Then for any $f\in\Rat_d(K)$, any $n\in\bN^*$, and any $r\in(0,1]$, we have
\begin{align*}
\bigl|L_n(f,r)-L(f)\bigr|
\le 2(2d-2)^2\cdot\left(L(f)+16\log M(f)+\sup_{\bP^1}|g_f|+|\log r|\right)\frac{\sigma_2(n)}{d_n}.
\end{align*}
\end{theorem}

\subsection{The key algebraic estimates and the end of the proof of Theorem~\ref{tm:approx}}

Theorem \ref{tm:approx} follows from Theorem \ref{tm:approx1} and 
the following.

\begin{lemma}\label{badbutexplicit}
Let $K$ be an algebraically closed field of characteristic $0$
that is complete with respect to a non-trivial absolute value $|\cdot|$,
and fix an integer $d>1$. Then
for every $f\in\Rat_d(K)$ and every lift $F$ of $f$, we have
\begin{eqnarray*}
\log|\Res(f)| & = &\log|\Res F|-2d\log|F|\le 2A_1(d,K),\\
 |L(f)| & \leq & -\log|d|-2\log|\Res F|+4d\log|F|\\
 && +\left\{\begin{array}{ll}
 A_2(d,K) & \text{when} \ K \ \text{is archimedean}\\
 2\log|F|-\log|C_F| & \text{when} \ K \ \text{is non-archimedean}
 \end{array}\right.\\
 \log M(f) & \leq & -2\log|d|-4\log|\Res F|+8d\log|F|+2A_2(d,K),\\
\sup_{\sP^1}|g_f|
& \le & \frac{2d-1}{d(2d-2)}\bigl(2d\log|F|-\log|\Res F|\bigr)
+\frac{3d-2}{d(d-1)}A_1(d,K),
\end{eqnarray*}
where the constants $A_1(d,K)$ and $A_2(d,K)$ are given by Lemmas \ref{lm:growthlift}
and \ref{th:upperlyap} and $C_F\in K\setminus\{0\}$ is the 
leading coefficient of $\det DF(1,z)\in K[z]$.
\end{lemma}

\begin{proof}
Fix $f\in\Rat_d(K)$ and a lift $F$ of $f$. The first estimate
follows from Lemma~\ref{lm:growthlift} and the definition of $\Res(f)$.
We now prove the second inequality. By Lemma~\ref{th:upperlyap}, we have
\begin{gather*}
 L(f):=\int_{\sP^1}\log(f^\#)\mu_f\le-\log|d|-2\log|\Res F|+4d\log|F|+A_2(d,K).
\end{gather*}
When $L(f)\ge 0$  (in particular when $K$ is archimedean),  the conclusion follows.
When $L(f)<0$, the field $K$ is non-archimedean and, by Proposition~\ref{tm:nonarchLyap}, and by
the lower estimate on $g_F$ from Lemma~\ref{lm:growthlift}, and the first inequality, we also have
\begin{align*}
 L(f) &\ge -\log|d|-\left(1-\frac{1}{d}\right)(-2\log|\Res F|+4d\log|F|)-(2\log|F|-\log|C_F|)\\
 &\ge\log|d|-(-2\log|\Res F|+4d\log|F|)-(2\log|F|-\log|C_F|),
\end{align*}
since $|d|\leq1$ and $-2\log|\Res F|+4d\log|F|\geq0$ in this case, which completes the proof.

When $K$ is non-archimedean, since $\log M(f)=-2\log|\Res(f)|=-2(\log|\Res F|-2d\log|F|)\ge 0$, the third estimate is obvious. When $K$ is archimedean,
since $\log M(f)=2\log\sup_{\sP^1}(f^\#)$, the third estimate follows
from Lemma~\ref{th:upperlyap}.

Finally, we prove the last estimate. By Lemma~\ref{lm:growthlift} and \eqref{eq:differencegreen}, and using that $-\log|\Res F|+2d\log|F|+2A_1(d,K)\ge 0$, we have
\begin{align*}
\left|g_f-\frac{1}{d}A_1(d,K)\right|\leq \frac{(2d-1)}{2d(d-1)}(2d\log|F|-\log|\Res F|+2A_1(d,K))\quad\text{on }\sP^1,
\end{align*}
ending the proof.
\end{proof}

\begin{proof}[Proof of Theorem~\ref{tm:approx}]
According to Theorem~\ref{tm:approx1}, it is sufficient to bound
\[|L(f)|+\log M(f)+\sup_{\sP^1}|g_f|\]
from above. Since $K$ is non-archimedean, we have $\log M(f)=-2\log|\mathrm{Res}(f)|$ and the first and last inequalities from Lemma~\ref{badbutexplicit} give
\[\sup_{\sP^1}|g_f|\leq -\frac{2d-1}{d(2d-2)}\log|\mathrm{Res}(f)|.\]
This ends the proof by Theorem~\ref{tm:approx1}.
\end{proof}

\subsection{Degeneration of the Lyapunov exponent}

Fix an integer $d>1$. Recall that
the integral domain $\cO(\mathbb{D}_K)[t^{-1}]$ consists of
all meromorphic functions on $\bD_K$ having no poles on $\bD_K^*$.

\begin{proof}[Proof of Theorem $\ref{tm:degenerate}$]
Let $f_t(z)\in\Rat_d(\cO(\mathbb{D}_K)[t^{-1}])$ 
be a meromorphic family of
rational functions (over $K$) of degree $d$ parametrized by $\mathbb{D}_K$,
and fix a lift $F_t=(F_{t,1},F_{t,2})$ of $f_t$ with $F_{t,i}\in \cO(\mathbb{D}_K)[t^{-1}][p_1,p_2]$.
Note that the function $t\mapsto C_{F_t}$ belongs to $\mathcal{O}(\mathbb{D}_K)[t^{-1}]$ and is not identically $0$. In particular, there exists $\delta>0$ such that $C_{F_t}\neq0$ for all $0<|t|<\delta$.

Fix $n\in\bN^*$ and $0<|t|<\delta$. By Theorem \ref{tm:approx}, we have
\begin{align*}
|L_n(f_t,\epsilon_{d^n})-L(f_t)| & \leq 8(d-1)^2\Bigg(|L(f_t)|-\frac{4d^2-2d-1}{d(2d-2)}\log|\Res(f_t)|+|\log\epsilon_{d^n}|\Bigg)\frac{\sigma_2(n)}{d_n}\\
& = 8(d-1)^2\Bigg(|L(f_t)|+\frac{4d^2-2d-1}{d(2d-2)}
(2d\log|F_t|-\log|\Res F_t|)\\
&\hspace*{0.5cm}+|\log\epsilon_{d^n}|\Bigg)\frac{\sigma_2(n)}{d_n},
\end{align*}
where we used that $\log|\mathrm{Res}(f_t)|=\log|\mathrm{Res}(F_t)|-2d\log|F_t|$, by Lemma~\ref{badbutexplicit}. Again by Lemma~\ref{badbutexplicit}, we have
\[|L(f)|\leq -\log|d|-2\log|\mathrm{Res}(F_t)|+(4d+2)\log|F_t|-\log|C_{F_t}|.\]
Note that, since the coefficients of $F_t$ lie in $\mathcal{O}(\mathbb{D}_K)[t^{-1}]$, it is clear that $\log|\Res F_t|=O(\log|t|^{-1})$, $\log|F_t|=O(\log|t|^{-1})$ and $\log|C_{F_t}|=O(\log|t|^{-1})$
as $t\to 0$. Moreover, for every $n\in\bN^*$, 
the limit
$\alpha_n:=\lim_{t\to 0}L_n(f_t,\epsilon_{d^n})/(\log|t|^{-1})\in\bN$
also exists by \eqref{eq:heighttruncate}.
Hence there is $C\ge 0$ such that for every $n\in\bN^*$,
\begin{gather*}
\limsup_{t\to 0}\frac{L(f_t)}{\log|t|^{-1}}
-C\cdot\frac{\sigma_2(n)}{d_n}
\le \alpha_n
\le\liminf_{t\to 0}\frac{L(f_t)}{\log|t|^{-1}}
+C\cdot\frac{\sigma_2(n)}{d_n}.
\end{gather*}
In particular, as $n\to\infty$, we find
\begin{gather*}
\limsup_{t\to 0}\frac{L(f_t)}{\log|t|^{-1}}
\le\liminf_{n\to\infty}\alpha_n\le\limsup_{n\to\infty}\alpha_n
\le\liminf_{t\to 0}\frac{L(f_t)}{\log|t|^{-1}}<\infty
\end{gather*}
since $\sigma_2(n)=O(n^2\log\log n)=o(d_n)$ as $n\to\infty$, ending the proof. 
\end{proof}

\section{The key global algebraic estimates}

Let $k$ be a global field and fix an integer $d>1$. For the sake of completeness, we include a proof of the following.
\begin{lemma}\label{critheight-Lyap}
For every $f\in\Rat_d(\overline{k})$, 
\begin{align*}
\sum_{v\in M_{k'}}N_vL(f)_v=\sum_{c\in\mathrm{Crit}(f)}\hat{h}_{f,k'}(c)=[k':k]h_{\Crit,k}(f),
\end{align*}
where $k'$ is any algebraic extension of $k$ such that $f\in k'(z)$.
\end{lemma}

\begin{proof}
Fix $f\in\Rat_d(\overline{k})$ and a lift $F$ of $f$, let $k',k''$
be any algebraic extensions of $k$ such that $f\in k'(z)$ and  
$\mathrm{Crit}(f)\subset\mathbb{P}^1_{k''}$, respectively, and
choose points $C_1^F,\ldots,C_{2d-2}^F\in k''^2\setminus\{0\}$
 such that $\det DF$ factors as 
$\det DF(p_0,p_1)=\prod_{j=1}^{2d-2}((p_0,p_1)\wedge C_j^F)$
in $k''[p_0,p_1]$. By \eqref{eq:Lyapdef} and the product formula, we have
\begin{align*}
 \sum_{v\in M_{k'}}N_vL(f)_v
=&\sum_{v\in M_{k'}}N_v\sum_{j=1}^{2d-2}(g_{f,v}(\pi(C_j^F))+\log\|C_j^F\|)\\
=&\frac{1}{[k'':k']}\sum_{j=1}^{2d-2}\sum_{v\in M_{k''}}N_v\left(g_{f,v}(\pi(C_j^F))+\log\|C_j^F\|\right).
\end{align*}
By Lemma~\ref{lm:growthlift} and the product formula, there is a finite subset
$E$ in $M_{k''}$ such that 
for every $v\in M_{k''}\setminus E$, we have $T_{F,v}\equiv 0$
on $\sP^1$, and $\|C_j^F\|_v=1$ for every $j\in\{1,\ldots,2d-2\}$. Hence
for every $j\in\{1,\ldots,2d-2\}$, recalling 
the definition \eqref{eq:defGreen} of $g_F$, we have
\begin{align*}
 \sum_{v\in M_{k''}}N_vg_{f,v}(\pi(C_j^F))
=&\sum_{v\in M_{k''}}N_vg_{F,v}(\pi(C_j^F))
=\sum_{v\in E}N_v\lim_{n\to\infty}\sum_{j=0}^{n-1}\frac{T_{F,v}(f^j(\pi(C_j^F))}{d^{j+1}}\\
=&\lim_{n\to\infty}\sum_{v\in E}N_v\sum_{j=0}^{n-1}\frac{T_{F,v}(f^j(\pi(C_j^F))}{d^{j+1}}
\\
=&\lim_{n\to\infty}\sum_{v\in M_{k''}}N_v\sum_{j=0}^{n-1}\frac{T_{F,v}(f^j(\pi(C_j^F))}{d^{j+1}}\\
=&\lim_{n\to\infty}\frac{1}{d^n}\sum_{v\in M_{k''}}N_v\log\|F^n(C_j^F)\|_v
-\sum_{v\in M_{k''}}N_v\log\|C_j^F\|_v,
\end{align*}
so that
\begin{gather*}
 \sum_{v\in M_{k''}}N_v\left(g_{f,v}(\pi(C_j^F))+\log\|C_j^F\|_v\right)
=\hat{h}_{f,k''}(\pi(C_j^F))=[k'':k']\cdot\hat{h}_{f,k'}(\pi(C_j^F)),
\end{gather*}
which completes the proof.
\end{proof}

Recall that the constants $A_1(d,K)$, $A_2(d,K)\geq0$ given by Lemmas~\ref{lm:growthlift} and~\ref{th:upperlyap} respectively, vanish when $K$ is non-archimedean, or depend only on $d$ but not on $K$, and can be computed explicitly when $K$ is archimedean. From now on, we denote by $A_1(d)$ (resp. $A_2(d)$) the constant $A_1(d,K)$ (resp. $A_2(d,K)$) for archimedean $K$.

Recall that a global field $k$ has no infinite places if $k$ is a function field,
and if $k$ is a number field, then 
$\sum_{v\in M_k:\text{infinite}}N_v=[k:\mathbb{Q}]$. We thus set
\begin{center}
$C(k):=\bigg\{\begin{array}{ll}
[k:\bQ] & $if $ k $ is a number field,$\\
0 & $if $ k $ is a function field.$ 
\end{array}$
\end{center}

\begin{lemma}\label{lm:absolutevalues-Lyap}
For every $f\in\mathrm{Rat}_d(\bar{k})$, we have
\begin{align*}
\frac{1}{[k':k]}\sum_{v\in M_{k'}}N_v|L(f)_v| &\leq (4d+2)\cdot h_{d,k}(f)
+C(k)\left(A_2(d)+\log(2d^3)\right)\\
\frac{1}{[k':k]}\sum_{v\in M_{k'}}N_v\log M(f)_v&\leq 8d\cdot h_{d,k}(f)+2C(k)A_2(d)\\
\frac{1}{[k':k]}\sum_{v\in M_{k'}}N_v\sup_{\sP^1(\bC_v)}|g_{f,v}|
& \le \frac{2d-1}{d-1}\cdot\ h_{d,k}(f)+\frac{3d-2}{d(d-1)}C(k)A_1(d),
\end{align*}
where $k'$ is any algebraic extension of $k$ such that $f\in k'(z)$.
\end{lemma}

\begin{proof}
Fix $f\in\Rat_d(\overline{k})$ and an algebraic extension $k'$ of $k$ such that $f\in k'(z)$. Fix a lift $F$ of $f\in\Rat_d(k')$ and let $C_F\in k'\setminus\{0\}$ be the leading coefficient of $\det DF (1, z)\in k'[z]$. First, note that, whenever $v$ is archimedean, we have
\[2\log|F|_v -\log|C_F|_v \geq -\log(2d^3).\]
Using Lemma \ref{badbutexplicit}, we have
\begin{align*}
|L(f)_v|  
\leq &-\log|d|_v-2\log|\Res F|_v+4d\log|F|_v\\
 &+\begin{cases}
 A_2(d) & \text{when }v\text{ is infinite}\\
 2\log|F|_v-\log|C_F|_v & \text{when }v\text{ is finite},
 \end{cases}\\
=&-\log|d|_v-2\log|\Res F|_v+(4d+2)\log|F|_v-\log|C_F|_v\\
 &+\begin{cases}
 -(2\log|F|_v-\log|C_F|_v)+A_2(d) & \text{when }v\text{ is infinite}\\
 0 & \text{when }v\text{ is finite},
 \end{cases}\\
\le&-\log|d|_v-2\log|\Res F|_v+(4d+2)\log|F|_v-\log|C_F|_v\\
 &+\begin{cases}
 \log(2d^3)+A_2(d) & \text{when }v\text{ is infinite}\\
 0 & \text{when }v\text{ is finite}.
 \end{cases}
\end{align*}
Taking the sum of the both sides over all $v\in M_{k'}$,
by the product formula, we obtain 
\[\sum_{v\in M_{k'}}N_v|L(f)_v|\le (4d+2)\sum_{v\in M_{k'}}N_v\log|F|_v+C(k')\left(A_2(d)+\log(2d^3)\right).\]
Since $C(k')=[k':\bQ]=[k':k][k:\bQ]=[k':k]C(k)$
when $k$ is a number field and since $\sum_{v\in M_{k'}}N_v\log|F|_v=[k':k]h_{d,k}(f)$, this gives the first estimate. The other two 
estimates also hold using similarly Lemma~\ref{badbutexplicit}. 
\end{proof}

\section{Applications over a function field}

In this section, we let $X$ be a smooth irreducible projective curve defined over $\bC$ and set $k=\bC(X)$. Fix an integer $d>1$. Recall that any place $v\in M_k$ is represented by a closed point $x$ of $X$ and that, in the present case, $N_v=1$ for all $v\in M_k$. Recall also that for any integers $d>1$ and $n\geq1$, we have $\epsilon_{d^n,v}=1$ for all places $v\in M_k$.

\subsection{The critical and multiplier heights over a function field}

Pick $f\in\mathrm{Rat}_d(k)$.
Note that, by definition, $\hat{h}_f(x)\ge 0$ and that $\hat{h}_f(x)=0$ if and only if $\deg(f^n(x))=O(1)$ as $n\to\infty$. Moreover, unless $f$ is isotrivial,
for every $x\in\mathbb{P}^1(\overline{k})$, 
we have $\hat{h}_{f,k}(x)=0$ if and only if the forward orbit of $x$ is finite, i.e. $x$ is preperiodic under iteration of $f$ (see, e.g., \cite{Benedetto10}).

\begin{proposition}\label{prop:fnctfield}
For every $f\in\Rat_d(k)$ and every integer $n\geq 1$, we have
\begin{gather*}
 \left|\frac{1}{nd_n}\max_{0\leq j\leq d_n}\deg(\sigma^*_{j,n}(f))-h_{\Crit,k}(f)\right|
\leq 8d(12d^2-8d-3)\frac{\sigma_2(n)}{d^n}\cdot h_{d,k}(f).
\end{gather*}
\end{proposition}

\begin{proof}
Let us set
\[E_n:=\frac{1}{nd_n}\max_{0\leq j\leq d_n}\{\deg(\sigma^*_{j,n}(f))\}-h_{\Crit,k}(f).\]
Recall that $h_{\Crit,k}(f)=\sum_{v\in M_k}L(f)_v$
by Lemma~\ref{critheight-Lyap} adn that
\[\max_{0\leq j\leq d_n}\deg(\sigma^*_{j,n}(f))=h_{\bP^{d_n},k}([\sigma^*_{0,n}(f):\cdots:\sigma^*_{d_n,n}(f)]).\]
Applying Theorem~\ref{tm:approx1} at each $v\in M_k$ and using the equality \eqref{eq:heighttruncate}, we get
\begin{align*}
\left|E_n\right|&=\left|\frac{1}{nd_n}h_{\bP^{d_n},k}([\sigma^*_{0,n}(f):\cdots:\sigma^*_{d_n,n}(f)])
-\sum_{v\in M_k}L(f)_v\right|\\
&=\left|\sum_{v\in M_k}\frac{1}{nd_n}
\log\max_{0\leq j\leq d_n}\{|\sigma^*_{j,n}(f)|_v\}
-\sum_{v\in M_k}L(f)_v\right|\\
&=\left|\sum_{v\in M_k}(L_{n}(f,1)_v-L(f)_v)\right|\\ & \leq 8(d-1)^2\sum_{v\in M_k}\biggl(|L(f)_v|+\log M(f)_v
+\sup_{\sP^1(\bC_v)}|g_{f,v}|\biggr)\frac{\sigma_2(n)}{d_n}.
\end{align*} 
By Lemma~\ref{lm:absolutevalues-Lyap} and $d_n=d^n-\sum_{k|n,k<n}d_k\geq(1-d^{-1})d^n$, this completes the proof.
\end{proof}

\begin{corollary}\label{cor:degcrit}
For every $f\in\Rat_d(k)$,
\begin{align*}
\liminf_{n\to\infty}\frac{1}{n}
\max\biggl\{\frac{\deg ((f^n)'(z))}{[k((f^n)'(z)):k]}: \, z\in\Fix^*(f^n)\biggr\}
\geq h_{\Crit,k}(f).
\end{align*}
\end{corollary}

\begin{proof}
For every $n\in\bN^*$, 
by Proposition~\ref{prop:fnctfield} and the definition of $L_{n,v}(f,1)$, we have
\begin{align*}
h_{\Crit,k}(f)+o(1)
& \le\frac{1}{nd_n}\max_{0\leq j\leq d_n}\{\deg(\sigma_{j,n}^*(f))\}=\sum_{v\in M_k}L_{n}(f,1)_v\\
&\leq\frac{1}{nd_n}\sum_{z\in\Fix^{*}(f^n)}\frac{\deg((f^n)'(z))}{[k((f^n)'(z)):k]}\\
&\le\frac{1}{nd_n}\cdot d_n\cdot\max\biggl\{\frac{\deg ((f^n)'(z))}{[k((f^n)'(z)):k]}: \, z\in\Fix^*(f^n)\biggr\},
\end{align*}
since $\Fix^{*}(f^n)$ contains exactly $d_n$ points. Together with the fact that
$\sigma_2(n)=o(d^n)$ and $d_n\sim d^n$, it is sufficient to make $n\to\infty$ to end the proof.
\end{proof}

\begin{example}
Let $f(z):=z^d+t\in\Rat_d(\bC(t))$, so that
$h_{d,\bC(t)}(f)=1$ and $\mathrm{Crit}(f)=\{0,\infty\}$. 
Since the multiplicity of $\infty$ as a critical point of $f$ equals
$d-1$ and $f(\infty)=\infty$, we have $\deg(f^n(\infty))=0$ for every $n\in\bN^*$
so that $\hat{h}_{f,\bC(t)}(\infty)=0$. 
On the other hand, for every $n\in\bN^*$, we have $\deg(f^n(0))=d^{n-1}$. Hence
\[
h_{\Crit,\bC(t)}(f)=(d-1)\hat{h}_{f,\bC(t)}(\infty)+
(d-1)\hat{h}_{f,\bC(t)}(0)=(d-1)\lim_{n\rightarrow\infty}\frac{\deg(f^n(0))}{d^n}=\frac{d-1}{d}.
\]
On the other hand, by Eremenko-Levin~\cite[Theorem 1.6]{eremenko-levin}, for every $n\geq1$, 
\begin{gather*}
 \min_{z_t\in\Fix^*(f_t^n)}\frac{1}{n}\log|(f_t^n)'(z_t)|=\frac{d-1}{d}\log|t|+o(1)
\quad\text{as }t\to\infty, 
\end{gather*}
which directly implies that for every $z\in\Fix^*(f^n)$,
\begin{gather*}
 \frac{1}{n}\frac{\deg((f^n)'(z))}{[k((f^n)'(z)):k]}=\frac{d-1}{d}.
\end{gather*}
In particular, the inequality in Corollary~\ref{cor:degcrit} is sharp in that
for any integer $d>1$, 
there exists $f\in\Rat_d(\mathbb{C}(t))$ such that
\begin{enumerate}
\item $h_{d,\bC(t)}(f)>0$ and $h_{\Crit,\mathbb{C}(t)}(f)>0$, and 
\item for every integer $n\geq1$ and every $z\in\Fix^*(f^n)$,
\[
\frac{1}{n}\frac{\deg((f^n)'(z))}{[k((f^n)'(z)):k]}=h_{\Crit,\bC(t)}(f).
\]
\end{enumerate}
\end{example}

\subsection{Proof of Theorem~\ref{mainthm:functionfield}}

The following is our main application of Theorem~\ref{tm:approx} over function fields. The equivalence between the first three points and the final one is new, and we deduce it from the other entries and Proposition~\ref{prop:fnctfield}.

\begin{theorem}\label{tm:9equivalent}
Let $X$ be an irreducible smooth projective curve defined over $\bC$ and let $k=\bC(X)$. Fix an integer $d>1$, let $f\in\Rat_d(k)$ and $S\subset M_k$ be the finite set of places of bad reduction of $f$.

If $f$ is not affine, the following assertions are equivalent:
\begin{enumerate}
\item $\max_{1\leq j\leq d_n}\deg(\sigma_{j,n}^*(f))=o(nd^n)$ as $n\rightarrow\infty$,
\item $\max_{1\leq j\leq d_n}\deg(\sigma_{j,n}^*(f))=O(n\sigma_2(n))$ as $n\rightarrow\infty$,
\item $\liminf_{n\to\infty}\frac{1}{n}\max\{[k((f^n)'(z)):k]^{-1}\deg((f^n)'(z)): z\in \Fix^*(f^n)\}=0$,
\item $\deg((f^n)'(z))=0$ for all $z\in \Fix^*(f^n)$ and all $n\in\bN^*$,
\item $h_{\Crit,k}(f)=0$,
\item for every $c\in\mathrm{Crit}(f)$, $\deg(f^n(c))=O(1)$ as $n\rightarrow\infty$,
\item $f$ is $J$-stable as a holomorphic family parametrized by $X\setminus S$,
\item $f$ is isotrivial.
\end{enumerate}
\end{theorem}

\begin{proof}
By Proposition~\ref{prop:fnctfield}, the assertions (i), (ii) and (v) are equivalent.  Moreover, items (v) and (vi) are equivalent by definition of the critical height.
Note that (iv) implies (iii) trivially and that (iii) implies (v) 
by Corollary~\ref{cor:degcrit}.
Now, if (viii) holds,
there exists a finite branched cover $Y\to X$ such that $k':=\bC(Y)$ is a finite extension of $k$ and
$M\in \mathrm{SL}_2(k')$ such that 
all the specializations $(M\circ f\circ M^{-1})_t\in\Rat_d(\bC)$ are independent of
$t\in Y$. In other words, $ f= M^{-1}\circ g \circ M$ where $g  \in \Rat_d(k')$ has coefficients in $\bC$. 
 In particular, for every $n\in\bN^*$, every $z\in \Fix^*(f^n)$ can be written as $M^{-1} y$ where $y\in \Fix^*(g^n)$, thus, by the chain rule, $(f^n)'(z)=  (g^n)'(y) \in\bC$, whence $\deg((f^n)'(z))=0$ for all $z\in\Fix^*(f^n)$, which is item (iv).

By e.g.~\cite[Theorem~1.4]{DeMarco-stable} item (v) is equivalent to the passivity of all critical points on $X\setminus S$, i.e. to (vii) by~\cite{MSS}. Finally, since we assumed that $f$ is not affine (vii) implies (viii) by Lemma 2.1 and Theorem 2.2 of \cite{McMullen4}.
\end{proof}

To conclude this section, we prove Theorem~\ref{mainthm:functionfield}.

\begin{proof}[Proof of Theorem~\ref{mainthm:functionfield}]
Note first that
\[h_{\bP^{d_n},k}\left([\sigma_{0,n}^*(f):\cdots:\sigma_{d_n,n}^*(f)]\right)=\max_{0\leq j\leq d_n}\deg\left(\sigma_{j,n}^*(f)\right),\]
so that Theorem~\ref{mainthm:functionfield} follows immediately from Theorem~\ref{tm:9equivalent}, once we check that a map $f$ which is affine satisfies assertions $(i)-(vii)$ of the above Theorem.

Pick $f:\mathbb{P}^1_k\rightarrow\mathbb{P}^1_k$ which is affine. Then there exists an elliptic curve $E$ defined over $k$, an isogeny $\phi:E\rightarrow E$ given by $\phi(P)=[m](P+P_0)$ for some $P_0\in E$ and some $m\in\mathbb{Z}$ with $m^2=d$ and a double branched cover $\pi:E\to\bP^1_k$ such that $\pi\circ \phi=f\circ\pi$ on $E$.
Then, for all $n\in\mathbb{N}^*$,$(f^n)'(z)=m^n$ for all periodic point $z\in\Fix^*(f^n)$, whence that assertions (i)--(iv) are satisfied. 

Finally, for all $c\in\mathrm{Crit}(f)$, the orbit of $c$ is finite whence (v)--(vii) hold trivially.
\end{proof}

\subsection{The case of function fields of higher dimensional varieties}\label{sec:higherdim}
Assume that $X$ is a normal irreducible projective variety defined over $\bC$ of any dimension and let $k:=\bC(X)$ be its function field. 
The results exposed in the present section can be adapted to this context. Let us explain how following the exposition of~\cite[Chapter 1]{BombieriGubler}.

\medskip

First, pick an ample line bundle $L$ on $X$ and denote by $\deg(Z)$ the degree of a cycle $Z$ with respect to $L$. For any $g\in k^\times$, let $\mathrm{ord}_Z(g)$ be the order of $g$ at $Z$ and let 
\[|g|_Z:=e^{-\deg(Z)\mathrm{ord}_Z(g)}, \ g\in \bC(X)^\times.\]
Note that for two distinct prime divisors $Z$ and $Z'$, the absolute values $|\cdot|_Z$ and $|\cdot|_{Z'}$ are non-trivial and not equivalent.
In addition, since the degree of a principal divisor is $0$, we have $\sum_{Z}\deg(Z)\mathrm{ord}_Z(g)=0$ for all $g\in k^\times$, where the sum ranges over all prime divisors of $X$, whence the sum is actually finite.
The field $k$ equipped with $(|\cdot|_Z)_Z$ and and $N_Z=1$ for all $Z$ is thus a global field.
 Equivalently, we have the product formula
\[\prod_{Z}|g|_Z=1, \ g\in k^\times.\]
We then may define the naive height function $h_{\bP^N,k}$ on $\bP^N(k)$ by
\[h_{\bP^N,k}(x):=-\sum_{Z}\deg(Z)\min_{0\leq j\leq N}\mathrm{ord}_Z(x_j)\]
if $x=[x_0:\cdots:x_N]\in \bP^N(k)$.
In particular, if  $x\in k^\times$ is a rational function on $X$, then
\[h_{\bP^1,k}(x)=h_{\bP^1,k}([x:1])=-\sum_{Z}\deg(Z)\min\big(0,\mathrm{ord}_Z(x)\big).\]
It thus can be computed as the degree of the divisor $x^{-1}\{\infty\}$. 
In particular, $h_{\mathbb{P}^1,k}(x)=0$ if and only if $x$ is constant.

~

Let us come back to our dynamical setting: the adaptation of Theorem~\ref{tm:9equivalent} remains valid in the present context, since we did not use any where that $X$ has dimension $1$ in the proof.

In particular, we have the

\begin{theorem}
let $X$ be a normal irreducible projective variety defined over $\bC$ of any dimension and let $k:=\bC(X)$ be its function field. Pick a rational map $f\in\mathrm{Rat}_d(k)$ of degree $d>1$ which is not affine. The following assertions are equivalent:
\begin{enumerate}
\item $h_{\bP^{d_n},k}([1,\sigma_{1,n}^*(f):\cdots:\sigma_{d_n,n}^*(f)])=o(nd^n)$ as $n\to\infty$,
\item $h_{\bP^{d_n},k}([1,\sigma_{1,n}^*(f):\cdots:\sigma_{d_n,n}^*(f)])=O(n\sigma_2(n))$ as $n\to\infty$,
\item $h_{\mathrm{crit},k}(f)=0$,
\item $f$ is isotrivial.
\end{enumerate}
\end{theorem}

\section{Multiplier maps and bifurcation currents}

\subsection{Symmetric powers of $\mathbb{P}^1$}
Pick an integer $N\geq2$. The symmetric group $\mathfrak{S}_N$ acts on 
$(\mathbb{P}^1)^N$ coordinate-wise: for each $\sigma\in\mathfrak{S}_N$, and each $(x_1,\ldots,x_N)$, we have
$\sigma\cdot (x_1,\ldots,x_N)=(x_{\sigma(1)},\ldots,x_{\sigma(N)})$. We denote by $S^{N}\mathbb{P}^1$ the $N$-th symmetric power of the projective line, i.e. the quotient of $(\mathbb{P}^1)^N$ by the action of the symmetric group $\mathfrak{S}_N$ by permutation. This quotient is a projective variety.

Denote by $P_N:(\mathbb{P}^1)^N\rightarrow S^N\mathbb{P}^1$ the quotient map and by $\pi_i:(\mathbb{P}^1)^N\rightarrow\mathbb{P}^1$ the projection onto the $i$-th coordinate.

Let $H_N$ be the divisor on $S^N\mathbb{P}^1$ such that $P_N^*H_N=\sum_{i=1}^N\pi_i^*\mathcal{O}_{\mathbb{P}^1}(1)$. This divisor is ample and we let $h_{S^N\mathbb{P}^1}$ be the associated height function on $(S^N\bP^1)(\bar{\bQ})$ (which is well-defined up to an additive constant), i.e. we let
\[h_{S^N\mathbb{P}^1,\mathbb{Q}}(P_N(x_1,\ldots,x_N)):=\sum_{j=1}^Nh_{\mathbb{P}^1,\bQ}(x_j),\]
for all $(x_1,\ldots,x_N)\in(\mathbb{P}^1(\bar{\mathbb{Q}}))^N$.
When working over the field $\mathbb{C}$, we also let $\omega_N$ be the positive smooth $(1,1)$-form representing $c_1(H_N)$ which satisfies
\[P_N^*\omega_N=\sum_{j=1}^N\pi_j^*(\lambda_{\mathbb{S}^1}),\]
where $\lambda_{\mathbb{S}^1}$ is the curvature form of the standard metrization of $\mathcal{O}_{\mathbb{P}^1}(1)$ associated with the height $h_{\mathbb{P}^1,\bQ}$, i.e. the Lebesgue measure on the unit circle normalized by $\lambda_{\mathbb{S}^1}(\bC)=1$.

We include a proof of the following for the sake of completeness.

\begin{lemma}
The variety $S^N\mathbb{P}^1$ is isomorphic to $\mathbb{P}^N$ and $H_N\simeq \mathcal{O}_{\mathbb{P}^N}(1)$. In particular,
\[h_{S^N\mathbb{P}^1,\mathbb{Q}}=h_{\mathbb{P}^N,\bQ}+O(1) \ \text{on} \ \mathbb{P}^N(\bar{\mathbb{Q}}).\]
\end{lemma}

\begin{proof}
For every $0\leq j\leq N$, set
\[
\eta_{j}([z_1:t_1],\ldots,[z_N:t_N]):=\sum_I\prod_{\ell=1}^Nz_\ell^{I_\ell} \cdot t_\ell^{1-I_\ell}\quad\text{on }(\bP^1)^N,
\]
where the sum is taken over all $N$-tuples $I\in\{0,1\}^N$
with $I_1+\cdots+I_N=N-j$. 
The map $\tilde{P}_N=[\eta_0,\ldots,\eta_N]:(\bP^1)^N\to \bP^N$ 
has the topological degree $N!$ and is surjective,
and each fiber of $\tilde{P}_N$ is invariant under the action of $\mathfrak{S}_N$. 
Hence $\tilde{P}_N$ descends as an isomorphism $S^N\mathbb{P}^1\to\mathbb{P}^N$.

To see that $H_N\simeq\mathcal{O}_{\mathbb{P} ^N}(1)$, it is sufficient to prove that $(H_N^N)=1$ since the ample cone of $\mathbb{P}^N$ is an open half line. By the change of variable formula and Fubini, we compute
\begin{align*}
(H_N^N) & =\int_{\mathbb{P}^N(\mathbb{C})}\omega_N^N=\int_{(\mathbb{P}^1(\mathbb{C}))^N}\frac{1}{N!}(P_N^*\omega_N)^N\\
& =\frac{1}{N!}\int_{(\mathbb{P}^1(\mathbb{C}))^N}\left(\sum_{j=1}^N\pi_j^*\lambda_{\mathbb{S}^1}\right)^N=\int_{(\mathbb{P}^1(\mathbb{C}))^N}\bigwedge_{j=1}^N\pi_j^*\lambda_{\mathbb{S}^1}\\
& = \prod_{j=1}^N\int_{\mathbb{P}^1(\mathbb{C})}\lambda_{\mathbb{S}^1}=1,
\end{align*}
as required. Finally, the comparison of height functions follows from fonctoriality properties of height functions~\cite{HS}.
\end{proof}

\begin{remark}\label{rem:affinecase}
Fix a field $K$. When $x_1,\ldots,x_N\in\mathbb{A}^1(K)$, it is easy to see that
\[P_N(x_1,\ldots,x_N)=[\sigma_N:\cdots:\sigma_1:1]\in \mathbb{A}^N(K),\]
where $\sigma_j$ is the degree $j$ elementary symmetric polynomial in $(x_1,\ldots,x_N)$ (recall that $\mathbb{A}^N(K)$ denotes the affine space).
\end{remark}

\subsection{Currents, DSH functions and the bifurcation currents}\label{sec:currentDSH}

We refer to~\cite[Appendix A]{dinhsibony2} for more details on currents and DSH functions. Pick any affine variety $X$ defined over $\bC$ and choose any embedding $\iota:X\rightarrow \mathbb{C}^N$. Let $D:=\mathcal{O}_{\mathbb{P}^N}(1)|_X$ and $\omega_D$ be the restriction of the ambient normalized Fubini-Study form to $X$. For any positive closed current $T$ of bidimension $(k,k)$ defined on $X$ and any Borel set $A \subset X$, we denote by $\|T\|_{A,D}$ the number
\[ 
\|T\|_{A,D} :=\int_A T \wedge \omega_D^k.
\]
This is the \emph{mass} of the current $T$ in $A$ relatively to $\omega_D$. 

Let $\Psi$ be a $(\ell,\ell)$-form in $X$. We say that $\Psi$ is DSH if we can write $dd^c\Psi=T^+-T^-$  where $T^\pm$ are positive closed currents of finite mass in $X$. We also set
\[
\|\Psi\|^*_{\mathrm{DSH},D}:=\inf\left(\|T^+\|_{X,D}+\|T^-\|_{X,D}\right),
\]
where the infimum is taken over all closed positive currents $T^\pm$ such that $dd^c\Psi=T^+-T^-$ (note that $\|T^+\|_{X,D}=\|T^-\|_{X,D}$ since they are cohomologuous).

This is not exactly the usual DSH norm but just a semi-norm. 
Nevertheless, one has 
$\|\Psi\|_{\mathrm{DSH}}^* \leq \|\Psi\|_{\mathrm{DSH}}$, where $\|\Psi\|_{\mathrm{DSH}}:=\|\Psi\|_{\mathrm{DSH}}^*+\|\Psi\|_{L^1}$. The interest of those DSH-norms lies in the fact that they behave nicely under change of coordinates. 
Furthermore, when $\Psi$ is $\mathcal{C}^2$ with support in a compact set $K$, there is a constant $C>0$ depending only on $K$ such that $\|\Psi\|_{\mathrm{DSH}}\leq C\|\Psi\|_{\mathcal{C}^2}$.

\medskip

For all $R>1$, we define the auxiliary function $\Psi_R:\mathbb{A}^N(\bC)\rightarrow\mathbb{R}$ by letting
\[\Psi_R(Z):=\frac{1}{\log R}\min\left\{\log\max\{\|Z\|,R\}-2\log R,0\right\}\]
for all $Z\in\mathbb{A}^N(\mathbb{C})$. For all $1\leq j\leq k:=\dim X$, define the $(k-j,k-j)$-DSH form on $X$
\[\Phi_{R,X,j}:=\Psi_R|_X\cdot\left(\omega_D\right)^{k-j}.\]

\begin{lemma}\label{lm:goodDSH}
The $(k-j,k-j)$-form $\Phi_{R,X,j}$ is DSH and continuous on $X$ with support in $\bar{\mathbb{B}}(0,R^2)\cap X$ and if $dd^c\Phi_{R,X,j}=T^+-T^-$, we have
\[\|T^{\pm}\|_{X,D}\leq\frac{\deg_D(X)}{\log R}.\]
\end{lemma}

\begin{proof}
First, note that the function $\Psi_R$ is DSH and continuous on $X$ and satisfies
\[\left\{\begin{array}{ll}
\Psi_R([f])=-1 & \text{ if } \ \iota([f])\in\mathbb{B}(0,R),\\
-1\leq \Psi_R([f])\leq 0 & \text{ if } \ \iota([f])\in\mathbb{B}(0,R^2)\setminus\bar{\mathbb{B}}(0,R),\\
0 & \text{ if } \ \iota([f])\notin \bar{\mathbb{B}}(0,R^2).
\end{array}\right.\]
As $\omega_{D}$ is positive and closed form on $X$, this implies $\Phi_{R,X,j}$ is DSH continuous with support in $\bar{\mathbb{B}}(0,R^2)\cap X$. Moreover, if $dd^c\Phi_{R,X,j}=T_+-T_-$, then 
\[T_\pm=(S_\pm|_X)\wedge \omega_D^{k-j},\]
where $S_\pm$ are such that $dd^c\Psi_R=S_\pm$ and, using B\'ezout's theorem, we have
\[\|T_\pm\|_{X,D}=\|(S_\pm|_X)\wedge\omega_D^{k-j}\|_{X,D}\leq\|S_\pm\|\cdot \deg_D(X),\]
where the mass $\|S_\pm\|=\int S_\pm\wedge \omega_{\mathrm{FS},\bP^N}^{N-1}$ is computed with respect to the Fubini Study form of $\mathbb{P}^N$. By construction of $\Psi_R$, we immediately get
 \[\|S_\pm\|=\int_{\mathbb{A}^N(\mathbb{C})}S_\pm\wedge\omega_D^{N-1}=\frac{1}{R},\]
 as required.
\end{proof}

Since two conjugated rational maps have the same Lyapunov exponent, the Lyapunov function $f\mapsto L(f)$ on $\mathrm{Rat}_d(\mathbb{C})$
{descends to}
a continuous and psh function $L:\mathcal{M}_d\rightarrow\mathbb{R}$; for more details including the complex analytic properties of the quotient map $\Rat_d(\bC)\to\mathcal{M}_d(\bC)$, we refer to \cite[\S6]{BB1}.
\begin{definition}\label{mu_and_T}
For any integer $1\leq p\leq 2d-3$, 
the $p$-\emph{bifurcation current} on $\mathcal{M}_d$ is given by
\[T_{\mathrm{bif}}^p:=(dd^cL)^p,\]
 and the \emph{bifurcation measure} on $\mathcal{M}_d$ is by $\mu_{\mathrm{bif}}:=T_{\mathrm{bif}}^{2d-2}=(dd^cL)^{2d-2}$. Finally we simply let $T_{\mathrm{bif}}:=T_{\mathrm{bif}}^1$.
\end{definition}

The measure $\mu_\mathrm{bif}$ is a finite positive measure on $\mathcal{M}_d$ of
strictly positive total mass (see \cite{BB1}). We also mention some other remarkable facts below, which will not be used explicitly in this article.

~
 
By DeMarco \cite{DeMarco2}, the support of $T_{\mathrm{bif}}$ coincides with the bifurcation locus of the moduli space $\mathcal{M}_d(\mathbb{C})$.  {If $p>1$,} the current $T_{\mathrm{bif}}^p$ detects, in a certain sense, stronger bifurcations than $T_{\mathrm{bif}}$. 
Indeed, its topological support admits several dynamical characterizations 
{similar to} that of the bifurcation locus: for example, it is the closure of parameters admitting $p$ distinct neutral cycles or $p$ critical points preperiodic to repelling cycles (see~\cite{BB1,Dujardin2012,gauthier:Indiana}). For more precise (even quantitative) equidistribution results, see \cite{GOV}.

\subsection{Complex analytic properties of the multiplier maps}

We now pick an integer $n\geq1$.

\begin{definition}
The period $n$ \emph{multiplier map} $\Lambda_n:\mathrm{Rat}_d\rightarrow \mathbb{P}^{d^n+1}$ is defined by
\[\Lambda_n([f]):=P_{d^n+1}\left((f^n)'(z_1),\ldots,(f^n)'(z_{d^n+1})\right),\]
for any representative $f$ of $[f]$, where $z_1,\ldots,z_{d^n+1}$ is any ordering of $\mathrm{Fix}(f^n)$.

Similarly, the period $n$ \emph{formally multiplier map} $\tilde{\Lambda}_n:\mathrm{Rat}_d\rightarrow \mathbb{P}^{d_n}$ is defined by
\[\tilde{\Lambda}_n([f]):=P_{d_n}\left((f^n)'(z_1),\ldots,(f^n)'(z_{d_n})\right),\]
for any representative $f$ of $[f]$, where $z_1,\ldots,z_{d_n}$ is any ordering of $\mathrm{Fix}^{*}(f^n)$.
\end{definition}

As mentioned before, the maps $\Lambda_n$ and $\tilde\Lambda_n$ descend to maps on $\mathcal{M}_d$ that are defined over $\mathbb{Q}$ and we still denote them by $\Lambda_n$ and $\tilde\Lambda_n$ respectively. Remark~\ref{rem:affinecase} implies that for any field $K$, any $n\geq1$ and any $[f]\in\mathcal{M}_d(K)$, we have
\[\Lambda_n([f])=[\sigma_{d^n+1,n}([f]):\cdots:\sigma_{1,n}([f]):1]\in\mathbb{A}^{d^n+1}(K)\]
and
\[\tilde\Lambda_n([f])=[\sigma^*_{d_n,n}([f]):\cdots:\sigma^*_{1,n}([f]):1]\in\mathbb{A}^{d_n}(K).\]
In particular, the maps $\Lambda_n$ and $\tilde{\Lambda}_n$ actually define morphisms from $\mathcal{M}_d$ to the corresponding affine space.

Let us begin with a few lemmas using classical pluripotential theory. Similarly to $L$, for every $r > 0$, the function $L_n(\cdot,r)$ descends to a continuous and psh function on $\mathcal{M}_d(\bC)$.
\begin{lemma}\label{lm:Lambdan_and_Ln}
For any integer $n\geq1$, we have
\[\frac{1}{nd_n}\tilde{\Lambda}_n^*(\omega_{d_n})=dd^cL_n(\cdot,1) \ \text{and} \ \frac{1}{n(d^n+1)}\Lambda_n^*(\omega_{d^n+1})=\sum_{k|n}\frac{d_k}{d^n+1}dd^cL_k(\cdot,1).\]
In particular, for any $1\leq j\leq 2d-2$ and any $n\geq1$, we have
$(\Lambda_n^*\omega_{d^n+1})^j\geq (\tilde\Lambda_n^*\omega_{d_n})^j$.
\end{lemma}

\begin{proof}
The first equality follows from the definition of the form $\omega_{d_n}$. To get the second equality, it is now sufficient to remark that
$\Lambda_n^*(\omega_{d^n+1})=\sum_{k|n}\frac{n}{k}\tilde \Lambda_k^*(\omega_{d_k})$ and to recall that, by definition of $d_n$, we have $d^n+1=\sum_{k|n}d_k$.

Using the two equalities, the continuity and the plurisubharmonicity of the $L_k$'s gives
\begin{align*}
(\Lambda_n^*(\omega_{d^n+1}))^j=\left(\sum_{k|n}nd_kdd^cL_k\right)^j\geq \left(nd_ndd^cL_n\right)^j=(\tilde\Lambda_n^*\omega_{d_n})^j,
\end{align*}
see, e.g. \cite[Corollary 3.4.9]{Klimek}.
\end{proof}

\begin{lemma}\label{lm:maximality}
Let $d>1$. Pick integers $n\geq1$ and $1\leq k\leq 2d-2$. Let $X\subset\mathcal{M}_d(\mathbb{C})$ be an irreducible subvariety of dimension $k$. Let $X_n\subset X$ (resp. $\tilde{X}_n\subset X$) be the algebraic subvariety of all $[f]\in X$ such that $\dim\left(X\cap \Lambda_n^{-1}\{\Lambda_n([f])\}\right)>0$ (resp. $\dim\left(X\cap \tilde\Lambda_n^{-1}\{\tilde\Lambda_n([f])\}\right)>0$). Then, we have
\[\int_{X_n}(\Lambda_n^*\omega_{d^n+1})^k=\int_{\tilde X_n}(\tilde\Lambda_n^*\omega_{d_n})^k=0.\]
\end{lemma}

\begin{proof}
Fix an integer $n\geq1$ and let $\tilde E_n\subset X$ be the set of $[f]\in X$ such that there exists an analytic map $\varphi:\mathbb{D}\rightarrow X$ with $\varphi(0)=[f]$ and such that  $L_n(\varphi(\cdot),1):\mathbb{D}\rightarrow\mathbb{R}$ is harmonic.

It is clear that, by definition of $\tilde E_n$, since $nd_ndd^cL_n(\cdot,1)=\tilde\Lambda_n^*\omega_{d_n}$, we have
\[\{[f]\in X\, : \, \dim X\cap \tilde \Lambda_n^{-1}\{\tilde \Lambda_n([f])\}>0\}\subset \tilde E_n.\]
By Lemma~\ref{lm:Lambdan_and_Ln}, the positive measure $\tilde\mu_n=(\tilde\Lambda_n^*\omega_{d_n})^k|_X$
has continuous potential on $X$. In particular, we have $\tilde\mu_n(\tilde E_n)=0$, by e.g.~\cite[Corollary A.10.3]{sibony} and the conclusion follows for $\tilde X_n$. The same proof also works for $X_n$.
\end{proof}

\begin{lemma}\label{lm:criterion}
Let $X\subset\mathcal{M}_d(\bC)$ be any complex algebraic subvareity of dimension $k\geq1$. Assume $X\cap\mathcal{L}_d(\bC)$ is a strict subvareity of $X$. Then 
\[\int_XT_{\mathrm{bif}}^k>0.\]
\end{lemma}

\begin{proof}
Let us set
\[\mathrm{Per}_n(w):=\{[f]\in\mathcal{M}_d(\mathbb{C})\, ; \ p_{d,n}([f],w)=0\}.\]
Up to taking a finite branched cover of $X$, we may assume that $(f_\lambda)_{\lambda\in X}$ is a family of rational maps endowed with $2d-2$ marked critical points, i.e. that there exist morphisms $c_1,\ldots,c_{2d-2}:X\to\bP^1$ such that $\mathrm{Crit}(f_\lambda)=\{c_1(\lambda),\ldots,c_{2d-2}(\lambda)\}$ counted with multiplicity. Since $\dim X>0$ and $X\cap\mathcal{L}_d(\bC)$ is a strict subvariety of $X$, the bifurcation locus of the quasi-projective vareity $Y:=X\setminus X\cap\mathcal{L}_d(\bC)$ is non-empty by \cite[Theorem 2.2]{McMullen4}. 

	By~\cite{MSS}, this implies that there exists $m_1\geq1$ and $\theta_1\in\mathbb{R}\setminus\mathbb{Q}$ such that $\mathrm{Per}_{m_1}(e^{2i\pi\theta_1})$ is a non-empty proper subvariety of $Y$. We repeat the argument to find $m_2>m_1$ and $\theta_2\in\mathbb{R}\setminus\mathbb{Q}$ such that $\mathrm{Per}_{m_2}(e^{2i\pi\theta_2})\cap \mathrm{Per}_{m_1}(e^{2i\pi\theta_1})\neq\varnothing$ has codimension $2$ in $X$. Applying this argument inductively gives $m_1<\cdots<m_k$ and $\theta_1,\ldots,\theta_k\in\mathbb{R}\setminus\mathbb{Q}$ such that
	$\mathrm{Per}_{m_1}(e^{2i\pi\theta_1})\cap\cdots\cap\mathrm{Per}_{m_k}(e^{2i\pi\theta_k})$ is a non-empty finite subset of $Y$.
	
	We define a quasi-projective curve in $Y$ by
	$C:=\bigcap_{j=1}^{k-1}\mathrm{Per}_{m_j}(e^{2i\pi\theta_j})$.
	As there exists a parameter $\tilde{\lambda}_0\in C$ for which $f_{\tilde{\lambda}_0}$ has a neutral cycle, non-persistent in $C$, the bifurcation locus $\mathrm{Bif}(C)=\supp(dd^cL|_C)$ of $C$ is non-empty. 
	Pick $\lambda_0\in \mathrm{Bif}(C)$ and let $U$ be a small neighborhood of that $\lambda_0$ in $X$. By Montel Theorem, there exists $\lambda_1\in U\cap C$, $1\leq i_1\leq 2d-2$ and $q_1,r_1\geq1$ such that $f_{\lambda_1}^{q_1}(c_{i_1}(\lambda_1))=f_{\lambda_1}^{q_1+r_1}(c_{i_1}(\lambda_1))$ and $f_{\lambda_1}^{q_1}(c_{i_1}(\lambda_1))$ is a repelling periodic point of $f_{\lambda_1}$ of exact period $r_1$, and such that this relation is not persistent through $C$ (see e.g.~\cite[Lemma 2.3]{favredujardin}). Let
	\[C_1:=\{\lambda\in U\, ; \ f_\lambda^{q_1}(c_{i_1}(\lambda))=f_{\lambda}^{q_1+r_1}(c_{i_1}(\lambda))\}\cap\bigcap_{j=1}^{k-2}\mathrm{Per}_{m_j}(e^{2i\pi\theta_j})\subset U.\]
	By the same argument as above, we find $\lambda_2\in U\cap C_1$ very close to $\lambda_1$, $1\leq i_2\leq 2d-2$ distinct from $i_1$, $q_2\geq1$ and $r_2>r_1$ such that $f_{\lambda_2}^{q_2}(c_{i_2}(\lambda_2))=f_{\lambda_2}^{q_2+r_2}(c_{i_2}(\lambda_2))$ and $f_{\lambda_2}^{q_2}(c_{i_2}(\lambda_2))$ is a repelling periodic point of $f_{\lambda_2}$ of exact period $r_2$, and such that this relation is not persistent through $C_1$. By a finite induction, we find a parameter $\lambda_{k-1}\in Y$, positive integers $q_1,\ldots,q_{k-1}$ and $r_1<r_2<\ldots<r_{k-1}$ and pairwise distinct indices $i_1,\ldots,i_{k-1}$ such that for all $1\leq j\leq k-1$:
	\[\lambda_{k-1}\in Y_j:=\{\lambda\in Y\, ; \ f_{\lambda}^{q_j}(c_{i_j}(\lambda))=f_{\lambda}^{q_j+r_j}(c_{i_j}(\lambda))\},\]
 $f_{\lambda_{k-1}}^{q_j}(c_{i_j}(\lambda_{k-1}))$ is a repelling periodic point for $f_{\lambda_{k-1}}$ of exact periods $r_j$ and the intersection of the $Y_j$'s is proper at $\lambda_{k-1}$. 
	Moreover, the intersection $Y_1\cap\cdots \cap Y_{k-1}\cap\mathrm{Per}_{m_k}(e^{2i\pi\theta_k})$ is proper, hence we may proceed as above to find $\lambda_k\in Y_1\cap\cdots \cap Y_{k-1}$, positive integers $q_k$ and $r_k$ and $i_k\notin \{i_1,\ldots,i_{k-1}\}$ such that 
	\[\lambda_k\in Y_k:=\{\lambda\in Y\, ; \ f_{\lambda}^{q_k}(c_{i_k}(\lambda))=f_{\lambda}^{q_k+r_k}(c_{i_k}(\lambda))\}\]
	and, for all $1\leq j\leq k-1$, $f_{\lambda_{k}}^{q_j}(c_{i_j}(\lambda_k))$ is a repelling periodic point for $f_{\lambda_{k}}$ of exact periods $r_j$ and the intersection of the $Y_j$'s is proper at $\lambda_{k}$. 
	By \cite[Theorem 6.2]{Article1}, we have $\lambda_k\in\supp(T_{\mathrm{bif}}^k\wedge[X])$ and the proof is complete.
\end{proof}

\subsection{Multiplier maps and the mass of bifurcation currents}
From now on, we fix an embedding $i:\mathcal{M}_d\hookrightarrow\mathbb{A}^N$ into some affine space and we let $\overline{\mathcal{M}_d}$ be the Zariski closure of $i(\mathcal{M}_d)$ in $\mathbb{P}^N$. Set
\[D:=\mathcal{O}_{\mathbb{P}^N}(1)|_{\overline{\mathcal{M}_d}}\]
and let $\omega_{\mathcal{M}_d,D}$ be the form $\omega_{\mathrm{FS},\mathbb{P}^N}|_{\overline{\mathcal{M}_d}(\bC)}$, where $\omega_{\mathrm{FS},\mathbb{P}^N}$ is the normalized Fubini-Study form of $\mathbb{P}^N(\mathbb{C})$. Recall that $D$ is an ample line bundle on $\overline{\mathcal{M}_d}$ and that we then have
\[(D^{2d-2})=\int_{\overline{\mathcal{M}_d}}\omega_{\mathcal{M}_d,D}^{2d-2}=\deg_D(\mathcal{M}_d).\]
The multiplier map $\Lambda_n$ and exact multiplier map $\tilde{\Lambda}_n$ define rational mappings $\overline{\mathcal{M}_d}\dashrightarrow \mathbb{P}^{d^n+1}$ $\overline{\mathcal{M}_d}\dashrightarrow \mathbb{P}^{d_n}$ which are defined over $\mathbb{Q}$ that we still denote by $\Lambda_n$ and $\tilde{\Lambda}_n$ respectively.

We prove here complex analytic properties of the multiplier maps that we will rely on in the sequel.
The first ingredient we use follows from~\cite{GOV} and can be summarized as follows.

\begin{proposition}\label{lm:mass}
There exists a constant $C>0$, depending only on $d$ and $\deg_D(\mathcal{M}_d)$, such that for any irreducible subvariety $X\subset\mathcal{M}_d(\mathbb{C})$ of dimension $k\geq1$ and any integers 
$1\leq j\leq k$ and $n\geq1$, we have
\[\left|\frac{1}{(nd_n)^j}\|(\tilde{\Lambda}_n^*\omega_{d_n})^j\|_{X,D}-\|T_{\mathrm{bif}}^j\|_{X,D}\right|\leq\left(\|T_{\mathrm{bif}}\|_{X,D}+C\deg_D(X)\frac{\sigma_2(n)}{d^n}\right)^j-\|T_{\mathrm{bif}}\|_{X,D}^{j},\]
and 
\begin{multline*}
\left|\frac{1}{n^j(d^n+1)^j}\|(\Lambda_n^*\omega_{d^n+1})^j\|_{X,D}-\|T_{\mathrm{bif}}^j\|_{X,D}\right|\\
\leq\left(\|T_{\mathrm{bif}}\|_{X,D}+C\deg_D(X)\frac{\sum_{\ell | n}\sigma_2(\ell)}{d^{n}}\right)^j-\|T_{\mathrm{bif}}\|_{X,D}^{j}.
\end{multline*}
\end{proposition}

For any $h\in\mathrm{Rat}_d(\mathbb{C})$, we let
\[\tilde{u}(h):=L(h)+32\log\left(\sup_{\mathbb{P}^1(\mathbb{C})}h^\#\right)+\sup_{\mathbb{P}^1(\mathbb{C})}|g_h|\geq0,\]
and, for any $[f]\in\mathcal{M}_d(\mathbb{C})$, we let $u([f]):=\inf \tilde{u}(h)$,
where the infimum is taken over all $h\in\mathrm{Rat}_d(\mathbb{C})$ which are conjugated to $f$.

\begin{claim}
There exists $C_1,C_2>0$ depending only on $d$ and $\deg_L(\mathcal{M}_d)$ such that
\begin{align}
u([f])\leq C_1\log^+\|[f]|\|_{\mathbb{A}^N(\bC)}+C_2,\label{eq:growthu}
\end{align}
for all $[f]\in\mathcal{M}_d(\mathbb{C})$.
\end{claim}

The proof of this Claim is an adaption of the argument in the proof of \cite[Lemma 6.5]{GOV} and we postpone it at the end of the paragraph.

\begin{proof}[Proof of Proposition~\ref{lm:mass}]
We follow closely the proof of~\cite[Theorem~A and Theorem~C]{GOV}. We treat both inequalities at the same time, since it follows the same strategy. 
Recall the definition of the auxiliary function $\Psi_R$ on $\bA^N(\bC)$
for any $R>0$ in Subsection~\ref{sec:currentDSH}.
Fix $X$, hence $k$, and $1\leq j\leq k$ as required.

\medskip

As before, for all $R>0$, we define $\Psi_R:\mathbb{A}^N\rightarrow\mathbb{R}$ by letting
\[\Psi_R(Z):=\frac{1}{\log R}\min\left\{\log\max\{\|Z\|,R\}-2\log R,0\right\}\]
for all $Z\in\mathbb{A}^N(\mathbb{C})$.
Fix $X$, $k\geq1$, and $1\leq j\leq k$ as required. Pick an integer $n$ and let
\[\tilde S_{n,j}:=\frac{1}{(nd_n)^j}(\tilde\Lambda_n^*\omega_{d_n})^j-T_{\mathrm{bif}}^j \ \text{and} \ \Phi_{R,X,j}:=\left.\left(\Psi_R\cdot\left(\omega_{\mathcal{M}_d,L}\right)^{k-j}\right)\right|_X, \ R>0.\]

Assume first $j=1$. By \eqref{eq:growthu} and by Theorem~\ref{prop:approxArch} and Lemmas~\ref{lm:goodDSH} and~\ref{lm:Lambdan_and_Ln}, we get
\[\left|\left\langle \tilde{S}_{n,1},\Phi_{R,X,1}\right\rangle\right|\leq \frac{\sigma_2(n)}{d_n}\left(C_1\log R+C_2\right)\|T^+\|_{X,D}\leq 2\frac{\sigma_2(n)}{d_n}\deg_L(X)\left(C_1+\frac{C_2}{\log R}\right),\]
where $T^\pm$ are closed and positive currents such that $dd^c\Phi_{R,X,1}=T^+-T_-$ and $C_1,C_2>0$ depend only on $d$ and $\deg_D(\mathcal{M}_d)$ but not on $R$.
Moreover, when $R\rightarrow+\infty$, 
\[\left|\left\langle \tilde{S}_{n,1},\Phi_{R,X,1}\right\rangle\right|\longrightarrow \left|\frac{1}{nd_n}\|\tilde{\Lambda}_n^*(\omega_{d_n})\|_{X,D}-\|T_{\mathrm{bif}}\|_{X,D}\right|.\]
so that
\begin{align}
\left|\frac{1}{nd_n}\|\tilde{\Lambda}_n^*(\omega_{d_n})\|_{X,D}-\|T_{\mathrm{bif}}\|_{X,D}\right|\leq C\frac{\sigma_2(n)}{d_n}\deg_D(X),\label{casej=1}
\end{align}
where $C>0$ is independent of $n$.
To conclude in that case, we just need to recall that $d_n\geq(1-d^{-1})d^n$.

Assume now $j\geq2$.
Lemma~\ref{lm:Lambdan_and_Ln} and an easy computation give
\[\tilde S_{n,j}=\left(dd^cL_n(\cdot,1)-dd^cL\right)\wedge\sum_{r=0}^{j-1}\frac{1}{(nd_n)^n}(\tilde{\Lambda}_n^*(\omega_{d_n}))^r\wedge T_{\mathrm{bif}}^{j-r-1}.\]
The same strategy as in the case $j=1$ implies that, for a given $R>0$, we have
\[\left|\langle \tilde S_{n,j},\Psi_{R,X,j}\rangle\right|\leq \frac{\sigma_2(n)}{d^{n}}(C_1\log R+C_2)\sum_{r=0}^{j-1}\left|\int_X(T_++T_-)\wedge\frac{1}{(nd_n)^r}(\tilde{\Lambda}_n^*\omega_{d_n})^r \wedge T_{\mathrm{bif}}^{j-r-1}\right|\]
 Applying B\'ezout's Theorem and Lemma~\ref{lm:goodDSH}, we get
\[\left|\langle \tilde S_{n,j},\Psi_{R,X,j}\rangle\right|\leq\frac{\sigma_2(n)}{d^{n}}\left(C_1+\frac{C_2}{\log R}\right)2\sum_{r=0}^{j-1}\deg_D(X)\frac{1}{(nd_n)^r}\|\tilde{\Lambda}_n^*\omega_{d_n}\|_{X,D}^r\|T_{\mathrm{bif}}\|_{X,D}^{j-r-1}.\]
Making $R\rightarrow+\infty$ , we find that
\begin{align}
\left|\frac{1}{(nd_n)^j}\|(\tilde{\Lambda}_n^*\omega_{d_n})^j\|_{X,D}-\|T_{\mathrm{bif}}^j\|_{X,D}\right|\label{massonsubvariety}
\end{align}
is bounded above by
\[C_4\frac{\sigma_2(n)}{d^{n}}\deg_D(X)\sum_{r=0}^{j-1}\frac{1}{(nd_n)^r}\|\tilde{\Lambda}_n^*\omega_{d_n}\|_{X,D}^r\|T_{\mathrm{bif}}\|_{X,D}^{j-r-1},\]
where $C_4=\max\{2C_1,C\}$ is again a universal constant. Finally, \eqref{casej=1} gives 
\[\frac{1}{(nd_n)^r}\|\tilde{\Lambda}_n^*\omega_{d_n}\|_{X,D}\leq \|T_{\mathrm{bif}}\|_{X,L}+C_4\deg_D(X)\frac{\sigma_2(n)}{d^{n}}.\]
In particular, the above upper bound on \eqref{massonsubvariety} can be rewritten as
\[\left(\|T_{\mathrm{bif}}\|_{X,D}+C_4\deg_D(X)\frac{\sigma_2(n)}{d^{n}}\right)^j-\|T_{\mathrm{bif}}\|_{X,D}^j,\]
and the proof is complete in that case.

~

To end the proof, observe that, if we let
\[S_{n,j}:=\frac{1}{n^j(d^n+1)^j}(\Lambda_n^*\omega_{d^n+1})^j-T_{\mathrm{bif}}^j,\]
since we have $\Lambda_n^*\omega_{d^n+1}=\sum_{\ell|n}\frac{n}{\ell}\tilde{\Lambda}_\ell^*(\omega_{d_\ell})$ and $d^n+1=\sum_{\ell|n}d_\ell$, we have
\[\frac{1}{n(d^n+1)}\Lambda_n^*\omega_{d^n+1}=\sum_{\ell |n}\frac{d_\ell}{n(d^n+1)}\frac{1}{d_\ell}\frac{n}{\ell}\tilde{\Lambda}_\ell^*(\omega_{d_\ell})\]
and an easy computation gives
\[S_{n,j}=\sum_{\ell |n}\frac{d_\ell}{d^n+1}\left(dd^cL_{\ell}(\cdot,1)-dd^cL\right)\wedge\sum_{r=0}^{j-1}\frac{1}{n^r(d^n+1)^r}(\Lambda_n^*\omega_{d^n+1})^r\wedge T_{\mathrm{bif}}^{j-r-1}.\]
We then apply the same strategy as in the previous case to conclude.
\end{proof}

\begin{proof}[Proof of the Claim]
This follows from the same strategy of proof as that of Lemma~6.5 of \cite{GOV}. 
Any $f\in\mathrm{Rat}_d(\bC)$ is conjugated to a (non-zero) finite number of rational maps $g$ admitting a lift $G$ of the form
\[G(z,w)=\left(Z^d+\sum_{j=1}^{d-1}a_jZ^jW^{d-j}+W^d,\sum_{j=1}^{d-1}b_jZ^jW^{d-j}\right).\]
with $(a_1,\ldots,a_{d-1},b_{d-1},\ldots,b_1)\in\mathbb{C}^{2d-2}$. Let $Z:=\{\mathrm{Res}(G)=0\}\subset\mathbb{P}^{2d-2}(\bC)$ and let $V:=\mathbb{P}^{2d-2}(\bC)\setminus Z$. This is a Zariski open set and, applying the same proof as in \cite[Lemma~6.5]{GOV} in the above family parametrized by $V$ gives $C_1,C_2,C_3>0$ such that
\[\tilde{u}(g)\leq C_1\left|\log|\mathrm{Res}(G)|\right|+C_2\log^+\|g\|+C_3\]
for all $g\in V$, where $\|g\|$ is the Euclidian norm of the vector $(a_1,\ldots,a_{d-1},b_{d-1},\ldots,b_1)\in\bC^{2d-2}$. Since $\mathrm{Res}\in\bC[V]$ and $V$ is a Zariski open subset of $\mathbb{P}^{2d-2}(\bC)$, this gives constants $C_4,C_5>0$ such that for all $g\in V$, we have
\[\tilde{u}(g)\leq -C_4\log d_{\bP^{2d-2}}(g,Z)+C_5.\]
The map $j:g\in V\mapsto [g]\in \mathcal{M}_d(\bC)$ extends as a meromorphic map $\mathbb{P}^{2d-2}(\bC)\dasharrow \overline{\mathcal{M}}_d(\bC)$ and is holomorphic on $V$. In particular,there exists $C,C',\alpha>0$ such that
\[d_{\mathbb{P}^N}(j(g),\partial\mathcal{M}_d) \leq C d_{\mathbb{P}^{2d-2}}(g,Z)^\alpha +C'\]
for all $g\in V$. This ends the proof by surjectivity of $j$ on $V$.
\end{proof}

\subsection{The multiplier map is generically finite for $n$ large enough}

We now come to our main result in this section. We want here to quantify the fact that multipliers determine the conjugacy class outside the flexible Latt\`es locus, up to finitely many choices.

Namely, we prove the following.

\begin{proposition}\label{tm:McMmap}
Fix an integer $d>1$. There exists an integer $n_0(d)\geq1$ depending only on $d$ such that the following holds
\begin{enumerate}
\item for all $n\geq n_0(d)$, the maps
$\Lambda_n$ and $\tilde\Lambda_n$ are generically finite-to-one,
\item for all $n\geq n_0(d)$, there exists a multiple $N$ of $n$ such that the map $\Lambda_N$ is finite-to-one outside of the flexible Latt\`es curve $\mathcal{L}_d\subset\mathcal{M}_d$.
\end{enumerate}
\end{proposition}

\begin{proof}
Let us begin with proving item $(i)$. We proceed by contradiction and assume that for arbitrary large $n\geq1$, the map $\tilde\Lambda_n$ (resp. $\Lambda_n$) is not finite-to-1 on any Zariski open subset of $\mathcal{M}_d(\mathbb{C})$.

In particular, we have $(\tilde{\Lambda}^*_n\omega_{d_n})^{2d-2}=0$ for all $n\geq1$ by Lemma~\ref{lm:maximality}. Applying Proposition~\ref{lm:mass} (to $X=\mathcal{M}_d$ and $k=2d-2$) , we thus get
\[\int_{\mathcal{M}_d}\mu_\mathrm{bif}\leq \left(\|T_{\mathrm{bif}}\|_{\mathcal{M}_d,D}+C\deg_D(\mathcal{M}_d)\frac{\sigma_2(n)}{d^{n}}\right)^{2d-2}-\|T_{\mathrm{bif}}\|_{\mathcal{M}_d,D}^{2d-2},\]
for arbitrary large $n\geq1$ where $C$ is a constant that depends only on $d$ and $\deg_D(\mathcal{M}_d)$. This is a contradiction, since the bifurcation measure has positive mass and the right hand side of the above inequality is $O(\sigma_2(n)d^{-n})$. Finally, if $(\tilde\Lambda_n^*\omega_{d_n})^{2d-2}$ is non-zero, by Lemma~\ref{lm:Lambdan_and_Ln}, we have $(\Lambda_n^*\omega_{d^n+1})^{2d-2}\geq (\tilde\Lambda_n^*\omega_{d_n})^{2d-2}>0$ and the conclusion follows again using Lemma~\ref{lm:maximality}.

~

We now come to the proof of item $(ii)$. By $(i)$, for any $n\geq n_0$, there exists a non-empty Zariski open set $U_n=\mathcal{M}_d\setminus Z_n$ of $\mathcal{M}_d(\mathbb{C})$ such that $\Lambda_n$ is finite-to-one on $U_n$. Let $Z_n:=\mathcal{M}_d(\mathbb{C})\setminus U_n$, it is clear that $Z_n$ is a proper subvariety of $\mathcal{M}_d(\mathbb{C})$ and that for all integers $n,k\geq1$, we have
\[\mathcal{L}_d\subset Z_n \ \text{ and } \ Z_{kn}\subset Z_n.\]
Fix now $n\geq n_0$. Let $Y_1,\ldots,Y_l$ be the irreducible components of $Z_n$ which are distinct from $\mathcal{L}_d$. Pick $1\leq j\leq l$ and let $k:=\dim Y_j\geq1$. As above, by Lemma~\ref{lm:maximality}, the map $\Lambda_{qn}$ is generically finite to one on $Y_j$ if and only if the measure $(\Lambda_{qn}^*\omega_{d^{qn}+1})^k$ is non-zero. Assume that for arbitrary large  $q\geq1$, the map $\Lambda_{qn}$ is not generically finite on $Y_j$. Lemma~\ref{lm:mass} gives
\[0\leq\int_{Y_j}T_\mathrm{bif}^k\leq C'\frac{\sigma_2(qn)}{d^{qn}}\]
for arbitrary large $q\geq 1$ and some constant $C'>0$ independent of $q$, whence $T_\mathrm{bif}^k=0$ on the variety $Y_j$. Since $Y_j\cap\mathcal{L}_d$ is a strict subvareity of $Y_j$, Lemma~\ref{lm:criterion} implies $T_\mathrm{bif}^k>0$ on $Y_j$. This is a contradiction.
\end{proof}

\section{The multiplier height}

Recall that we have fixed an embedding $\iota:\mathcal{M}_d\hookrightarrow \mathbb{A}^N$ into some affine space, that $\overline{\mathcal{M}}_d$ is the closure of $\mathcal{M}_d$ in $\mathbb{P}^N$ and that
$D=\mathcal{O}(1)|_{\overline{\mathcal{M}_d}}$ is an ample line bundle.
Recall the definition of the multiplier map $\Lambda_n$ and the constants $C_1(d,D)$ and $C_2(d,D)$ from the introduction and that these constants depend only on $d$ and on the embedding $\iota$ we chose.

In this section we prove the following.

\begin{theorem}\label{tm:multiplierheight}
There exists an integer $n_1\geq1$ depending only on $d$ such that for all $m\geq n_1$, there exists a multiple $n$ of $m$ such that the following holds: 
\[2C_1(d,D) h_{\mathcal{M}_d,D}-A_n\leq \frac{1}{n(d^n+1)}h_{\mathbb{P}^{d^n+1},\bQ}\circ \Lambda_n\leq 2C_2(d,D) h_{\mathcal{M}_d,D}+A_n\]
on $(\mathcal{M}_d\setminus\mathcal{L}_d)(\bar{\mathbb{Q}})$ for some constant $A_n\geq0$ depending only on $d$ and $n$.
\end{theorem}

\begin{proof}
Choose an integer $n_1\geq n_0(d)$ where $n_0$ is given by Proposition~\ref{tm:McMmap}, and such that the following holds for all $m\geq n_1$
\begin{align*}
\left\{\begin{array}{l}
C\deg_D(\mathcal{M}_d)\frac{\sum_{\ell | m}\sigma_2(\ell)}{d^{m}}<\frac{1}{2}\|T_{\mathrm{bif}}\|_{\mathcal{M}_d,D},\\
\left(\|T_{\mathrm{bif}}\|_{\mathcal{M}_d,D}+C\deg_D(\mathcal{M}_d)\frac{\sum_{\ell | m}\sigma_2(\ell)}{d^{m}}\right)^{2d-2} <\|T_{\mathrm{bif}}\|_{\mathcal{M}_d,D}^{2d-2}+ \frac{1}{2}\|\mu_{\mathrm{bif}}\|_{\mathcal{M}_d}.
\end{array}
\right.
\end{align*}
Choose now $m\geq n_1$ and let $n$ be the multiple of $m$ given by the second point of Proposition~\ref{tm:McMmap} so that the map $\Lambda_n$ is finite to one outside of the flexible Latt\`es locus $\mathcal{L}_d$.

Recall that $H_{d^n+1} \simeq \mathcal{O}_{\mathbb{P}^{d^n+1}}(1)$ and that $h_{S^{d^n+1}\mathbb{P}^1}$ is an ample height associated with $H_{d^n+1}$ and that $h_{S^{d^n+1}\mathbb{P}^1}=h_{\mathbb{P}^{d^n+1}}+O(1)$.
Let $E_n$ be the restriction of $H_{d^n+1}$ to the Zariski closure $X_n$ of $\Lambda_n(\mathcal{M}_d)$ in $\mathbb{P}^{d^n+1}$. By the choice of $n$, $\dim X_n=2d-2$ and, by definition, $E_n$ is an ample divisor on $X_n$.

\medskip

The morphism $\Lambda_n:\mathcal{M}_d\to \bP^{d^n+1}$ is generically finite,
and defines a rational map $\Lambda_n:\overline{\mathcal{M}_d}
\dashrightarrow X_n$,
which is dominant and generically finite, by the choice of $n$. Moreover, $\Lambda_n|_{\mathcal{M}_d}$ is a morphism. In particular, there exists a normal projective variety $\widehat{\mathcal{M}_d}$, a birational morphism $\psi:\widehat{\mathcal{M}_d}\rightarrow\overline{\mathcal{M}_d}$ and a generically finite morphism $\widehat{\Lambda}_n:\widehat{\mathcal{M}_d}\rightarrow X_n$ such that the following diagram commutes
\[\xymatrix{
\widehat{\mathcal{M}_d}\ar[d]_\psi \ar[rd]^{\widehat{\Lambda}_n}& \\
\overline{\mathcal{M}_d}\ar@{-->}[r]_{\Lambda_n}   & X_n}\]
and the map $\psi$ restricts as an isomorphism $(\widehat{\mathcal{M}_d})_{\mathrm{reg}}\to(\mathcal{M}_d)_{\mathrm{reg}}$.

As the divisors $D$ and $E_n$ are ample, and as $\widehat{\Lambda}_n$ and $\psi$ are morphisms, the divisors
\[\widehat{E}_n:=\widehat{\Lambda}_n^*E_n \ \text{ and } \ \widehat{D}:=\psi^*D\]
and big and nef divisors of $\widehat{\mathcal{M}_d}$. We have the

\begin{lemma}\label{lm:Siu}
Let $N_1,N_2\geq1$ be any integers satisfying 
\[(2d-2)N_2\left(2\deg_D(\mathcal{M}_d)\right)^{1/(2d-2)} \leq N_1\|\mu_{\mathrm{bif}}\|_{\mathcal{M}_d}^{1/(2d-2)}\]
and let $M_1,M_2\geq1$ be any integers satisfying  \[M_1\deg_D(\mathcal{M}_d)\geq3(d-1)M_2\|T_{\mathrm{bif}}\|_{\mathcal{M}_d,D} . \] Then the following divisors are big:
\[ n(d^{n}+1)M_1\widehat{D}-M_2\widehat{E}_n \ \text{ and } \ N_1\widehat{E}_n-n(d^n+1)N_2\widehat{D}.\]
\end{lemma}

Let us first finish the proof of Theorem~\ref{tm:multiplierheight}. Take $M_1$ and $M_2$ satisfying the hypothesis of the above lemma. We follow closely the strategy of \cite[\S 1]{silverman-heightestimate}: since $M_1n(d^{n}+1)\widehat{D}-M_2\widehat{E}_n$ is big, there exists an integer $k\geq1$ such that $kn(d^{n}+1)M_1\widehat{D}-kM_2\widehat{E}_n$ is effective. According to \cite[Theorem B.3.2(e)]{HS}, this implies 
\[h_{\widehat{\mathcal{M}}_d,kn(d^{n}+1)M_1\widehat{D}-kM_2\widehat{E}_n}(x)\geq O(1) \ \text{ for all }x\in\psi^{-1}\mathcal{M}_d(\bar{\mathbb{Q}}),\]
since the base locus of $kn(d^{n}+1)M_1\widehat{D}-kM_2\widehat{E}_n$ is (at worst) contained in $\widehat{\mathcal{M}_d}\setminus\psi^{-1}(\mathcal{M}_d)$. Furthermore, functorial properties of height functions \cite[Theorem B.3.2(b-c)]{HS} then imply
\begin{align*}
h_{\widehat{\mathcal{M}}_d,kn(d^{n}+1)M_1\widehat{D}-kM_2\widehat{E}_n}(x) & = h_{\widehat{\mathcal{M}}_d,kn(d^{n}+1)M_1\widehat{D}}(x)-h_{\widehat{\mathcal{M}}_d,kM_2\widehat{E}_n}(x)+O(1) \\
&= kn(d^{n}+1)M_1h_{\widehat{\mathcal{M}}_d,\widehat{D}}(x)-kM_2h_{\widehat{\mathcal{M}}_d,\widehat{E}_n}([f])+O(1) \\
&= kn(d^{n}+1)M_1h_{\widehat{\mathcal{M}}_d,\widehat{D}}(x)-kM_2h_{S^{d^n+1}\mathbb{P}^1}\circ\widehat{\Lambda}_n(x)+O(1)
\end{align*}
 for all $x\in \psi^{-1}(\mathcal{M}_d\setminus\mathcal{L}_d)(\bar{\mathbb{Q}})$, since $\widehat{\Lambda}_n$ is finite exactly on $\psi^{-1}(\mathcal{M}_d\setminus\mathcal{L}_d)$. All the above summarizes as
 \[M_1h_{\widehat{\mathcal{M}}_d,\widehat{D}}(x)\geq \frac{1}{n(d^n+1)}M_2h_{S^{d^n+1}\mathbb{P}^1}\circ\Lambda_n(\psi(x))+O(1) \ \text{ for all }x\in\psi^{-1}(\mathcal{M}_d\setminus\mathcal{L}_d)(\bar{\mathbb{Q}}).\]
Since $\psi$ is a surjective morphism, $\Lambda_n$ is well-defined and has finite fibers on $(\mathcal{M}_d\setminus\mathcal{L}_d)(\bar{\mathbb{Q}})$ and since $\widehat{\Lambda}_n=\Lambda\circ \psi$ on $\psi^{-1}((\mathcal{M}_d\setminus\mathcal{L}_d)(\bar{\mathbb{Q}}))$, the definition of $\widehat{D}$ implies
\[M_1h_{\mathcal{M}_d,D}([f])\geq \frac{1}{n(d^n+1)}M_2h_{S^{d^n+1}\mathbb{P}^1}\circ\Lambda_n([f])+O(1) \ \text{ for all }[f]\in(\mathcal{M}_d\setminus\mathcal{L}_d)(\bar{\mathbb{Q}}).\]
In particular, we choose $M_1$ and $M_2$ large enough so that
\[\frac{3(d-1)\|T_{\mathrm{bif}}\|_{\mathcal{M}_d,D}}{\deg_D(\mathcal{M}_d)}\leq \frac{M_1}{M_2}\leq\frac{4(d-1)\|T_{\mathrm{bif}}\|_{\mathcal{M}_d,D}}{\deg_D(\mathcal{M}_d)}\]
 and the fact that $h_{S^{d^n+1}\mathbb{P}^1}=h_{\mathbb{P}^{d^n+1}}+O(1)$ gives the wanted inequality.

~
 
Applying the same strategy to the big divisor $N_1 \widehat{E}_n-n(d^n+1)N_2L$ gives
\begin{align*}
N_1 h_{S^{d^n+1}\mathbb{P}^1}\circ\Lambda_n([f]) 
& \geq n(d^n+1)N_2h_{\mathcal{M}_d,D}([f])+O(1),
\end{align*}
 for all $[f]\in(\mathcal{M}_d\setminus\mathcal{L}_d)(\bar{\mathbb{Q}})$. As above, we choos $N_1$ and $N_2$ so that
\[2(d-1)\left(\frac{2\deg_D(\mathcal{M}_d)}{\|\mu_{\mathrm{bif}}\|_{\mathcal{M}_d}}\right)^{1/(2d-2)}\leq \frac{N_1}{N_2}\leq4(d-1)\left(\frac{\deg_D(\mathcal{M}_d)}{\|\mu_{\mathrm{bif}}\|_{\mathcal{M}_d}}\right)^{1/(2d-2)}\]
and the conclusion follows.
\end{proof}

It now remains to prove Lemma~\ref{lm:Siu}.
For that we rely on Siu's Theorem which states that if $D$ and $E$ are nef divisors on a projective variety $X$ of dimension $k$, and if
\begin{align}
(D^k)>k(D^{k-1}\cdot E),\label{siu}
\end{align}
then the divisor $D-E$ is big (see, e.g.~\cite[Theorem 2.2.15]{Laz}).

\begin{proof}[Proof of Lemma~\ref{lm:Siu}]
Observe first that the map $\widehat{\Lambda}_n=\Lambda_n\circ\psi:\widehat{\mathcal{M}}_d\rightarrow X_n$ is surjective and has finite fibers on a Zariski open set of $\widehat{\mathcal{M}}_d$. In particular, $\widehat{E}_n$ is big and nef. Moreover,
\begin{align*}
(\widehat{E}_n^{2d-2}) =\int_{\widehat{\mathcal{M}}_d}\left(\widehat{\Lambda}_n^*\omega_{d^n+1}\right)^{2d-2} &\geq\int_{\psi^{-1}(\mathcal{M}_d)}\left(\widehat{\Lambda}_n^*\omega_{d^n+1}\right)^{2d-2}= \int_{\mathcal{M}_d}\left(\Lambda_n^*\omega_{d^n+1}\right)^{2d-2}\\
& \geq \|(\Lambda_n^*\omega_{d^n+1})^{2d-2}\|_{\mathcal{M}_d}, 
\end{align*}where we used the change of variable formula and that $\deg(\psi)=1$.
According to Proposition~\ref{lm:mass}, the choice of $n$ guarantees that
\[\|(\Lambda_n^*\omega_{d^n+1})^{2d-2}\|_{\mathcal{M}_d}> \frac{1}{2}n^{2d-2}(d^n+1)^{2d-2}\|\mu_\mathrm{bif}\|_{\mathcal{M}_d},\]
so that the above gives
\[(\widehat{E}_n^{2d-2})>\frac{1}{2}n^{2d-2}(d^n+1)^{2d-2}\|\mu_\mathrm{bif}\|_{\mathcal{M}_d}.\]
On the other hand, since $\psi$ is surjective finite and proper and $\deg(\psi)=1$, $\widehat{D}$ is big and nef and the projection formula gives
\[(\widehat{D}^{2d-2})=(D^{2d-2})=\deg_D(\mathcal{M}_d).\]
Fix an embedding $j:\widehat{\mathcal{M}}_d\hookrightarrow \mathbb{P}^q$.

Recall that, if $T$ is a closed positive $(1,1)$-current (resp. form) on $\widehat{\mathcal{M}}_d$, it can be identified with the restriction of a closed positive $(1,1)$-current (resp. form) on $\mathbb{P}^q$, so that we may identify $\widehat{\Lambda}_n^*\omega_{d^n+1}$ with the restriction to $\widehat{\mathcal{M}}_d$ of a current $S_n$ and $\psi^*\omega_{\mathcal{M}_d,D}$ with the restriction of a form $\alpha$.

Let $\|T\|$ be the mass of any current $T$ with respect to the Fubini Study form $\omega$ of $\mathbb{P}^q$ and let $\deg(\widehat{\mathcal{M}}_d)$ be the degree of the embedding $j(\widehat{\mathcal{M}}_d)$. B\'ezout's Theorem gives
\[\deg_D(\mathcal{M}_d)=(\widehat{D}^{2d-2})=\deg(\widehat{\mathcal{M}}_d)\|\alpha\|^{2d-2},\]
and
\begin{align*}
(\widehat{E}_n^{2d-3}\cdot \widehat{D}) & = \deg(\widehat{\mathcal{M}}_d)\|S_n\|^{2d-3}\|\alpha\|,
\end{align*}
and
\[(\widehat{E}_n^{2d-2})=\deg(\widehat{\mathcal{M}}_d)\|S_n\|^{2d-2}.\]
In particular, this gives
\begin{align*}
(\widehat{E}_n^{2d-3}\cdot \widehat{D}) & = \deg_D(\mathcal{M}_d)\left(\frac{\|S_n\|}{\|\alpha\|}\right)^{2d-3}=\deg_D(\mathcal{M}_d)\left(\frac{(\widehat{E}_n^{2d-2})}{(\widehat{D}^{2d-2})}\right)^{1-1/(2d-2)}\\
& = \deg_D(\mathcal{M}_d)^{1/(2d-2)}(\widehat{E}_n^{2d-2})^{1-1/(2d-2)}.
\end{align*}
Pick any integers $N_1,N_2\geq1$ such that
\[(2d-2)N_2\left(2\deg_D(\mathcal{M}_d)\right)^{1/(2d-2)} \leq N_1\|\mu_{\mathrm{bif}}\|_{\mathcal{M}_d}^{1/(2d-2)}.\]
By the above, the divisors $N_1\widehat{E}_n$ and $N_2\widehat{D}$ satisfy
\begin{align*}
(2d-2)\left((N_1\widehat{E}_n)^{2d-3}\cdot n(d^n+1)N_2\widehat{D}\right) <\big((N_1\widehat{E}_n)^{2d-2}\big).
\end{align*}
By Siu's Theorem, this implies $N_1\widehat{E}_n-n(d^n+1)N_2\widehat{D}$ is big, as required.

~

We now treat the other case. Recall that, by the projection formula, we have
\[(\widehat{D}^{2d-2})=\deg_D(\mathcal{M}_d).\]
Moreover, the current $\widehat{\Lambda}^*_n\omega_{d^n+1}$ has
 continuous potentials and the form $\psi^*\omega_{\mathcal{M}_d,D}$ is smooth, so the measure $(\widehat{\Lambda}^*_n\omega_{d^n+1})\wedge(\psi^*\omega_{\mathcal{M}_d,D}^{2d-3})$ does not give mass to the divisor at infinity $\widehat{\mathcal{M}}_d\setminus\psi^{-1}(\mathcal{M}_d)$. By the change of variable formula, we get
\[(\widehat{D}^{2d-3}\cdot \widehat{E}_n)=\int_{\psi^{-1}(\mathcal{M}_d)}\widehat{\Lambda}_n^*(\omega_{d^n+1})\wedge (\psi^*\omega_{\mathcal{M}_d,D})^{2d-3}=\|\Lambda_n^*\omega_{d^n+1}\|_{\mathcal{M}_d,D},\]
and, as above, Proposition~\ref{lm:mass} and the choice of $n$ give
\[(\widehat{D}^{2d-3}\cdot \widehat{E}_n)<\frac{3}{2}n(d^n+1)\|T_\mathrm{bif}\|_{\mathcal{M}_d,D}.\]
Let $M_1,M_2\geq1$ be any integers such that $M_1\deg_D(\mathcal{M}_d)\geq3(d-1)M_2\|T_{\mathrm{bif}}\|_{\mathcal{M}_d,D}$ and let us set $\widehat{D}_1:=M_1n(d^n+1)\widehat{D}$ and $\widehat{E}_n':=M_2\widehat{E}_n$. As a consequence of the above, we find
\begin{align*}
(2d-2)\left(\left(\widehat{D}_1\right)^{2d-3}\cdot \widehat{E}_n'\right)& < (2d-2)M_1^{2d-3}n^{2d-2}(d^n+1)^{2d-2}M_2\frac{3}{2}\|T_\mathrm{bif}\|_{\mathcal{M}_d,D}\\
& < 3(d-1)M_1^{2d-3}M_2n^{2d-2}(d^n+1)^{2d-2}\|T_\mathrm{bif}\|_{\mathcal{M}_d,D}\\
& < M_1^{2d-2}n^{2d-2}(d^n+1)^{2d-2}\deg_D(\mathcal{M}_d)=(\widehat{D}_1^{2d-2}).
\end{align*}
As above, Siu's Theorem implies that $\widehat{D}_1-\widehat{E}_n'$ is big.
\end{proof}

\section{The critical height and applications}
To simplify the notations and since there is no possible ambiguity, we omit the subscript $\bQ$ for height functions in the whole section.
\subsection{Comparing the multiplier height with the critical height}

Our main goal here is to prove the equivalent of Proposition~\ref{prop:fnctfield} over number fields. Since there are both finite places having positive residue characteristic and infinite places, we get a less precise control on the error term.

\begin{proposition}\label{prop:nbfield}
There exists a constant $C_n$, depending only on $d$ and $n$ such that for any $f\in\mathrm{Rat}_d(\bar{\mathbb{Q}})$ and any $n\geq1$, we have
\[\left|\frac{1}{nd_n}h_{S^{d_n}\mathbb{P}^1}\circ\tilde\Lambda_n(f)- h_{\mathrm{crit}}(f)\right|\leq \bigg(1056d^2-1024d-24\bigg)(d-1) h_{d}(f)\cdot \frac{\sigma_2(n)}{d_n}+C_n.\]
Moreover, $C_n$ can be computed explicitly.
\end{proposition}

\begin{proof}
Pick $n\geq1$ and $f\in\mathrm{Rat}_d(\bar{\bQ})$. Let $K_n$ be a finite extension of $\mathbb{Q}$ such that $f\in \mathrm{Rat}_d(K_n)$ and $\mathrm{Fix}^{*}(f^n)\cup\mathrm{Crit}(f)\subset\mathbb{P}^1(K_n)$. 
By Lemma~\ref{critheight-Lyap}, we have
\begin{align*}
h_{\mathrm{crit}}(f)=\sum_{c\in\mathrm{Crit}(f)}\hat{h}_f(c)=\frac{1}{[K_n:\mathbb{Q}]}\sum_{v\in M_{K_n}}N_vL_v(f).
\end{align*}
As above, we find
\begin{align*}
E_n &:=\frac{1}{nd_n}\sum_{z\in\mathrm{Fix}^{*}(f^n)}h_{\mathbb{P}^1}((f^n)'(z))-h_{\mathrm{crit}}(f)=\frac{1}{[K_n:\mathbb{Q}]}\sum_{v\in M_{K_n}}N_v\left(L_n(f,1)_v-L_v(f)\right).
\end{align*}
Setting
\[\delta_n(f):=\frac{1}{[K_n:\mathbb{Q}]}\sum_{v\in M_{K_n}}\frac{N_v}{nd_n}\sum_{z\in\mathrm{Fix}^{*}(f^n)}\log\frac{\max\{|(f^n)'(z)|_v,1\}}{\max\{|(f^n)'(z)|_v,\varepsilon_{d^n,v}\}},\]
we find
\begin{align*}
E_n =\frac{1}{[K_n:\mathbb{Q}]}\sum_{v\in M_{K_n}}N_v\left(L_n(f,\varepsilon_{d^n,v})_v-L_v(f)\right)+\delta_n(f).
\end{align*}
We write $v\leq d^n$ if $v\in M_{K_n}$ is non-archimedean and if its residue characteristic $p$ satisfies $p\leq d^n$. Recall that $\varepsilon_{d^n,v}=\min\{|\ell|_v^{d^n}\, : \, 1\leq\ell\leq d^n\}$ satisfies $\varepsilon_{d^n,v}<1$ if and only if $v\leq d^n$ and that $\varepsilon_{d^n,v}>d^{-nd^n}$ in any case.

Since for all $v\in M_{K_n}$ and all $z\in\mathrm{Fix}^{*}(f^n)$, we have
\[\left|\log\max\{|(f^n)'(z)|_v,1\}-\log\max\{|(f^n)'(z)|_v,\varepsilon_{d^n,v}\}\right|\leq \log\varepsilon_{d^n,v}^{-1}\]
and since $\mathrm{Card}(\mathrm{Fix}^{*}(f^n))=d_n$ and since for all prime $p$, $[K_n:\mathbb{Q}]=\sum_{v|p}N_v$, we find
\begin{align*}
|\delta_n(f)| & \leq \frac{1}{[K_n:\mathbb{Q}]}\sum_{v\leq d^n}\frac{N_v}{nd_n}\sum_{z\in\mathrm{Fix}^{*}(f^n)}\log\varepsilon_{d^n,v}^{-1}\\
& \leq  \frac{1}{[K_n:\mathbb{Q}]}\sum_{v\leq d^n}\frac{N_v}{nd_n}\sum_{z\in\mathrm{Fix}^{*}(f^n)}nd^n\log d\leq \sum_{0<p\leq d^n}d^n\log d\leq d^{2n}\log d.
\end{align*}
Combining Theorem~\ref{tm:approx1} and Theorem~\ref{prop:approxArch} at all places of $K_n$ gives
\begin{align*}
[K_n:\mathbb{Q}]|E_n| &\leq \sum_{v\in M_{K_n}}N_v\left|L_n(f,\varepsilon_{d^n,v})_v-L(f)_v\right|+[K_n:\mathbb{Q}]d^{2n}\log d\\
& \leq \sum_{v\in M_{K_n}} AN_v\left(|L(f)_v|+\sup_{\mathsf{P}^{1}_{\mathbb{C}_v}}|g_{f,v}|_v+16\log M(f)_v-\log\varepsilon_{d^n,v}\right)\frac{\sigma_2(n)}{d_n}\\
& \hspace{0.5cm} +[K_n:\mathbb{Q}]d^{2n}\log d
\end{align*}
where $A:=8(d-1)^2>0$ and where we used that $M(f)_v\geq1$ for all $v$. On the one hand, reasoning as above, we find
\begin{align*}
\sum_{v\in M_{K_n}^0} -N_v\log(\varepsilon_{d^n,v})\frac{\sigma_2(n)}{d_n} & \leq [K_n:\mathbb{Q}]\frac{d}{d-1}n\sigma_2(n)d^n\log d.
\end{align*}
On the other hand, Lemma~\ref{lm:absolutevalues-Lyap} then gives
\begin{align*}
\sum_{v\in M_{K_n}}N_v\Big(|L(f)_v|+\sup_{\mathsf{P}^{1}_{\mathbb{C}_v}}|g_{f,v}|_v+16\log M(f)_v\Big)\leq & \frac{132d^2-128d-3}{d-1} h_{d,K_n}(f)\\
& +[K_n:\mathbb{Q}]C(d),
\end{align*}
with $C(d)=\frac{3d-2}{d(2d-2)}A_1(d)+33A_2(d)+\log(2d^3)$, where $A_1(d)$ and $A_2(d)$ are given by Lemma~\ref{lm:growthlift} and~\ref{th:upperlyap} respectively and depends only on $d$.
 Since $h_{d,K_n}(f)=[K_n:\mathbb{Q}]h_d(f)$, this concludes the proof.
\end{proof}

We now deduce easily the following.
\begin{corollary}\label{cor:overnumberfield}
There exists a constant $\tilde{C}_n>0$, depending only on $d$ and $n$ such that for any $f\in\mathrm{Rat}_d(\bar{\mathbb{Q}})$ and any $n\geq1$, we have
\[\left|\frac{1}{n(d^n+1)}h_{S^{d^n+1}\mathbb{P}^1}\big(\Lambda_n(f)\big)- h_{\mathrm{crit}}(f)\right|\leq \bigg(1056d^2-1024d-24\bigg)(d-1) h_d(f)\frac{\sum_{k|n}\sigma_2(k)}{d^n+1}+\tilde{C}_n\]
where $\tilde{C}_n$ can be computed explicitly.
\end{corollary}

\begin{proof}
Pick an integer $n$ and $f\in\mathrm{Rat}_d(\bar{\bQ})$.
By definition of $L_n(f,1)_v$, see \eqref{eq:truncate}, the chain rule, 
and since $d^n+1=\sum_{m|n}d_m$, we have
\begin{align*}
 R_n:=\frac{1}{n(d^n+1)}\sum_{z\in\mathrm{Fix}(f^n)}h_{\mathbb{P}^1}((f^n)'(z))-h_{\Crit}(f)=\sum_{m|n}\frac{d_m}{d^n+1}E_m,
\end{align*}
which with Proposition~\ref{prop:nbfield} completes the proof.
\end{proof}

\subsection{The minimal height, quantified}
Recall that we have chosen an embedding $\mathcal{M}_d\hookrightarrow\mathbb{A}^N$ and that we have set $D=\mathcal{O}_{\mathbb{P}^N}(1)|_{\iota(\mathcal{M}_d)}$. As in \cite[\S6.2]{silverman-moduli}, for any $[f]\in\mathcal{M}_d(\bar{\mathbb{Q}})$,
let
\[h^{\mathrm{min}}([f]):=\min\left\{h(g)\, ; g\sim f \ \text{and} \ g\in\mathrm{Rat}_d(\bar{\mathbb{Q}})\right\}.\]
According to \cite[Lemma~6.32]{silverman-moduli}, there exists constants $A_1,A_2>0$ and $B_1,B_2$ such that 
\begin{align}
A_1h_{\mathcal{M}_d,D}-B_1\leq h^{\mathrm{min}}\leq A_2h_{\mathcal{M}_d,D}+B_2 \ \text{on} \ \mathcal{M}_d(\bar{\mathbb{Q}}).\label{heightmin}
\end{align}
In fact, we only need the bound from above and we can quantify the multiplicative constant. Namely, we can prove the next lemma.

\begin{lemma}\label{lm:min}
There exists $A>0$ depending only on $d$ and $\deg_D(\mathcal{M}_d)$ such that
\[h^{\mathrm{min}}\leq (2d-2)h_{\mathcal{M}_d,D}+A, \ \ \text{on} \ \mathcal{M}_d(\bar{\mathbb{Q}}).\]
\end{lemma}

\begin{proof}
As in the proof of Lemma~\ref{lm:Siu}, we use the numerical criterion of Siu, see formula~\eqref{siu}. We proceed in two steps:
First, we treat the case when $f$ has at least three distinct fixed points. Second, we treat the case when $f$ has at least one multiple fixed point.

Let us define $X$ as the set of rational maps $f$ having a lift of the form
\[F(Z,W)=\left(Z^d+\sum_{j=1}^{d-1}a_jZ^jW^{d-j},\sum_{j=0}^{d-1}b_jZ^jW^{d-j}\right)\]
with $\sum_ja_j=\sum_\ell b_\ell$, so that $f$ has $0$, $1$ and $\infty$ as fixed points. This variety identifies with a Zariski open set of a linear subspace of $\mathbb{P}^{2d+1}$ of dimension $2d-2$, i.e. we may compactify $X$ as $\mathbb{P}^{2d-2}$. 

Let $\Pi:X\rightarrow\mathcal{M}_d$ be the restriction of the canonical projection. It is clear that the fiber of $\Pi$ over any $[f]\in\mathcal{M}_d\setminus\mathrm{Per}_1(1)$, is non-empty and finite with a uniform bound on its cardinality. In particular, $\Pi$ defines a rational map
\[\Pi:\mathbb{P}^{2d-2}\dashrightarrow \overline{\mathcal{M}}_d\]
which is dominant and generically finite. Let $Y$ be a normal projective variety defined over $\mathbb{Q}$ and $\pi_1,\pi_2$ be a morphism such that
\[\xymatrix{
Y\ar[d]_{\pi_1} \ar[rd]^{\pi_2}& \\
\mathbb{P}^{2d-2}\ar@{-->}[r]_{\Pi}   & \overline{\mathcal{M}_d}}\]
commutes and the map $\pi_1$ is an isomorphism onto $X$.
Set
\[H:=\pi_2^*D \ \text{and} \ E:=\pi_1^*\mathcal{O}_{\mathbb{P}^{2d-2}}(1).\]
As $\pi_1$ and $\pi_2$ are surjective morphisms which are generically finite, the divisors $H$ and $E$ are big and nef. Moreover, the projection formula gives
\[(H^{2d-2})=\deg(\pi_2)(D^{2d-2})=\deg(\Pi)\deg_D(\mathcal{M}_d)\]
and
\[(E^{2d-2})=\deg(\pi_1)(\mathcal{O}_{\mathbb{P}^{2d-2}}(1)^{2d-2})=1.\]
We now choose an embedding $Y\hookrightarrow \mathbb{P}^M$ and denote by $\deg(Y)$ the degree $\deg_{\mathcal{O}_{\mathbb{P}^M}(1)}(Y)$. 

As in the proof of Lemma~\ref{lm:Siu}, we may identify $\pi_2^*\omega_{\mathcal{M}_d,D}$ with the restriction to $Y$ of a form $\alpha$ and $\pi_1^*\omega_{2d-2}$ with the restriction of a current $S$. Bézout's Theorem gives
\[(H^{2d-2})=\int_{Y}\left(\pi_2^*\omega_{\mathcal{M}_d,D}\right)^{2d-2}=\deg(Y)\cdot\|\alpha\|^{2d-2}\]
and
\[(E^{2d-2})=\int_{Y}\left(\pi_1^*\omega_{2d-2}\right)^{2d-2}=\deg(Y)\cdot\|S\|^{2d-2}.\]
Combined with the above use of the projection formula, this gives
\[\left(\frac{\|S\|}{\|\alpha\|}\right)^{2d-2}\deg(\Pi)\deg_D(\mathcal{M}_d)=1.\]
Using again B\'ezout, we find
\begin{align*}
(H^{2d-3}\cdot E) & =\int_{Y}\left(\pi_2^*\omega_{\mathcal{M}_d,D}\right)^{2d-3}\wedge\left(\pi_1^*\omega_{2d-2}\right)\\
& =\deg(Y)\cdot\|\alpha\|^{2d-3}\|S\|\\
& =\deg(Y)\cdot\|\alpha\|^{2d-2}\frac{\|S\|}{\|\alpha\|}\\
& =\frac{\|S\|}{\|\alpha\|}(H^{2d-2})\\
& =\left(\deg(\Pi)\deg_D(\mathcal{M}_d)\right)^{-1/(2d-2)}\cdot (H^{2d-2}).
\end{align*}
We note that $\deg(\Pi)\deg_D(\mathcal{M}_d)>1$ and we define
\[\tilde{H}:=(2d-2)H.\]
All the above summarizes as
\begin{align*}
(2d-2)\left(\tilde{H}^{2d-3}\cdot E\right) & =(2d-2)^{2d-2}(H^{2d-3}\cdot E)\\
&<(2d-2)^{2d-2}\left(\deg(\Pi)\deg_D(\mathcal{M}_d)\right)^{1/(2d-2)}(H^{2d-3}\cdot E)\\
& <(2d-2)^{2d-2}(H^{2d-2})=(\tilde{H}^{2d-2}).
\end{align*}
By Siu's numerical criterion, this implies $\tilde{H}-E$ is a big divisor on $Y$. In particular, by an argument similar to that in the proof of Lemma~\ref{lm:Siu}, we find
\[(2d-2)h_{\mathcal{M}_d,D}\circ\Pi\geq h_{\mathbb{P}^{2d-2}}+O(1)=h_{d,\mathbb{Q}}+O(1)\]
on $X(\bar{\mathbb{Q}})$ and in turn we have the desired estimate
\[h^{\mathrm{min}}([f])\leq h_d(f)\leq (2d-2)h_{\mathcal{M}_d,D}([f])+O(1)\]
for all $f\in X(\bar{\mathbb{Q}})$.

~

We now come to the case of rational maps having a multiple fixed point. Let $Z$ be the set of rational maps $f$ having a lift of the form
\[F(z,w)=\left(Z^d+\sum_{j=1}^{d-1}a_jZ^jW^{d-j}+W^d,Z^{d-1}W+\sum_{j=1}^{d-2}b_jZ^jW^{d-j}\right).\]
The family $Z$ is again a Zariski open set of a linear subspace of $\mathbb{P}^{2d+1}$ of dimension $2d-3$. Moreover, any rational map with a multiple fixed point is conjugated to such a map $f\in Z$; and fibers of the restriction of the canonical projection $\Pi:Z\rightarrow\mathrm{Per}_1(1)\subset\mathcal{M}_d$ over any $[f]\in \mathrm{Per}_1(1)$ are non-empty, finite with uniformly bounded cardinality. The same proof as above gives
\begin{align*}
h^{\mathrm{min}}([f])\leq h(f)& \leq (2d-3)h_{\mathcal{M}_d,D}([f])+O(1)\\
& \leq (2d-2)h_{\mathcal{M}_d,D}([f])+O(1)
\end{align*}
for all $f\in Z(\bar{\mathbb{Q}})$.
Since any class $[f]\in\mathcal{M}_d(\bar{\mathbb{Q}})$ admits either a representative in $X(\bar{\mathbb{Q}})$ or a representative in $Z(\bar{\mathbb{Q}})$, this ends the proof.
\end{proof}

\subsection{The critical height is a moduli height, quantified}\label{sec:mainresult}
We now want to give the proof of Theorem~\ref{tm:critheight} and of Theorem~\ref{tm:McMexact}. 

\begin{proof}[Proof of Theorem~\ref{tm:critheight}]
 Corollary~\ref{cor:overnumberfield} implies that
\begin{align}
\left|\frac{1}{n(d^n+1)}h_{\mathbb{P}^{d^n+1}}\circ\Lambda_n([f])- h_{\mathrm{crit}}([f])\right|\leq C(d)\cdot h^{\mathrm{min}}([f])\cdot \frac{\sum_{k|n}\sigma_2(k)}{d^n}+C_n,\label{onMd}
\end{align}
 for all $n\geq1$ and all $[f]\in\mathcal{M}_d(\bar{\mathbb{Q}})$, where $C_n>0$ is a constant depending only on $d$ and $n$ and $C(d)=(1056d^2-1024d-24)(d-1)$.

Pick now $n$ so that Theorem~\ref{tm:multiplierheight} applies and such that
\[(2d-2)C(d)\frac{\sum_{k|n}\sigma_2(k)}{d^n}\leq C_1(d,D),\]
where $C_1(d,D)$ is defined in the introduction. Combining Lemma~\ref{lm:min} with \eqref{onMd} gives
\begin{align*}
\left|\frac{1}{n(d^n+1)}h_{\mathbb{P}^{d^n+1}}\circ\Lambda_n([f])- h_{\mathrm{crit}}([f])\right|\leq C_1(d,D) h_{\mathcal{M}_d,D}([f])+C_n',
\end{align*}
for all $[f]\in\mathcal{M}_d(\bar{\mathbb{Q}})$. Since $C_1(d,D)\leq C_2(d,D)$, Theorem~\ref{tm:multiplierheight} ends the proof.
\end{proof}

\begin{proof}[Proof of Theorem~\ref{tm:McMexact}]
Fix $n\in\mathbb{N}^*$. Le us first prove that, when $n$ is large enough, $h_{S^{d_n/n}\mathbb{P}^1}\circ \Lambda_{n,\mathrm{ex}}$ is a height function on $(\mathcal{M}_d\setminus\mathcal{L}_d)(\bar{\bQ})$. By Proposition~\ref{prop:nbfield} and Lemma~\ref{lm:min},
\begin{align*}
\left|\frac{1}{nd_n}h_{S^{d_n}\mathbb{P}^1}\circ \tilde{\Lambda}_n-h_{\mathrm{crit}}\right| & \leq (2d-2)C(d)\frac{\sigma_2(n)}{d^n}h_{\mathcal{M}_d,D}+C_n',
\end{align*}
where $C(d)=1056d^3-1024d^2-24d$ and $C_n'$ is a constant that depends only in $d$ and $n$.

By Theorem~\ref{tm:critheight}, for all $[f]\in(\mathcal{M}_d\setminus\mathcal{L}_d)(\bar{\mathbb{Q}})$, we have
\[C_1(d,D)h_{\mathcal{M}_d,D}-A\leq h_{\mathrm{crit}} \leq C_2(d,D)h_{\mathcal{M}_d,D}+A,\]
where $A$ depends only on $d$ and $C_1(d,D)$ and $C_2(d,D)$ are defined in the introduction.
In particular, as soon as 
\begin{align}
\frac{1}{2}C_1(d,D)> (2d-2)C(d)\frac{\sigma_2(n)}{d^n},\label{eq:goodinteger}
\end{align}
the function $\frac{1}{nd_n}h_{S^{d_n}\mathbb{P}^1}\circ \tilde{\Lambda}_n$ is a Weil height on $(\mathcal{M}_d\setminus\mathcal{L}_d)(\bar{\mathbb{Q}})$. More precisely, there exists a constant $C_n''$ such that
\[\frac{1}{2}C_1(d,D)h_{\mathcal{M}_d,D}-C_n''\leq \frac{1}{nd_n}h_{\bP^{d_n}}\circ\tilde{\Lambda}_n \leq \left(C_2(d,D)+\frac{1}{2}C_1(d,D)\right)h_{\mathcal{M}_d,D}+C_n'',\]
on $(\mathcal{M}_d\setminus\mathcal{L}_d)(\bar{\mathbb{Q}})$.

~

We now prove that $\tilde\Lambda_{n}$ is finite to one on $(\mathcal{M}_d\setminus\mathcal{L}_d)(\bar{\mathbb{Q}})$. Assume first there exists $x\in\mathbb{P}^{d_n}(\bar{\mathbb{Q}})$ such that there exists an irreducible component $X$ of $\Lambda_{n,\mathrm{ex}}^{-1}\{x\}$ with positive dimension and such that $X(\bar{\mathbb{Q}})\cap\mathcal{L}_d(\bar{\mathbb{Q}})$ is finite.
Let $V:=X\setminus X\cap\mathcal{L}_d$, so that $V$ is an irreducible quasi-projective variety of positive dimension defined over a number field $K$.
By the first step of the proof, we have
\[h_{\mathcal{M}_d,D}\leq \frac{2}{C_1(d,D)nd_n}\left(h_{\mathbb{P}^{d_n}}(x)+C_n'\right)<+\infty \ \text{on} \ V(\bar{K}).\]
On the other hand, using again the functorial properties of height functions, we see that the height $h_{\mathcal{M}_d,D}$ restricts to $V$ as a height function associated with an ample line bundle on $X$. This is a contradiction, since $h_{\mathcal{M}_d,D}$ is uniformly bounded on $V(\bar{K})$.

To conclude, it is sufficient to remark that, since $\Lambda_{n,\mathrm{ex}}$ has finite fibers on $(\mathcal{M}_d\setminus\mathcal{L}_d)(\bar{\mathbb{Q}})$, it also has finite fibers on $(\mathcal{M}_d\setminus\mathcal{L}_d)(K)$ for any extension $K$ (algebraic or not) of $\bar{\mathbb{Q}}$, ending the proof.
\end{proof}

\begin{remark}[Quadratic case]\normalfont
\begin{enumerate}
\item When $d=2$, the moduli space is canonically isomorphic to the affine space $\mathbb{A}^2$ and the isomorphism is given by $(\sigma_{1,1},\sigma_{1,2})$ by~\cite{Milnor3,silverman-spacerat}. We can choose $D$ to be the line bundle $\mathcal{O}(1)$. In particular, $\deg_D(\mathcal{M}_2)=1$ and the equation \eqref{eq:goodinteger} reads as
\[\frac{1}{4}\|\mu_{\mathrm{bif}}\|^{1/2}>8608\frac{\sigma_2(n)}{2^n}.\]
According to~\cite[Corollary~D]{GOV}, the mass of $\mu_{\mathrm{bif}}$ is then
\[\|\mu_{\mathrm{bif}}\|=\frac{1}{3}-\frac{1}{8}\sum_{n\geq1}\frac{\phi(n)}{(2^n-1)^2}\geq0.1875.\]
where $\phi$ is the Euler totient function. In particular, the condition is fulfilled for $n\geq27$, whence $\tilde{\Lambda}_n$ is finite to one for all $n\geq27$.
\item This is explicit but not satisfactory, since it is easy to prove that the map $\Lambda_{n}$ has finite fibers for \emph{all} $n\geq1$. Indeed, for $n=1$, the map $\tilde{\Lambda}_1=\Lambda_1$ is known to be an isomorphism from $\mathbb{A}^2$ to its image. Moreover, as seen in the proof of Proposition~\ref{tm:McMmap}, $\Lambda_n$ is finite to one on a Zariski open set $\mathbb{A}^2(\bC)\setminus Z_n$ for all $n\geq1$ and the set $Z_n$ satisfies $Z_n\subset Z_1=\emptyset$ for all $n>1$, as required.
\end{enumerate}
\end{remark}

\bibliographystyle{short}
\bibliography{biblio}

\end{document}